\documentclass[12pt,reqno,fleqn]{amsart}  
\usepackage{tikz}
\usepackage{amsmath,amstext,amsthm,amssymb,amsxtra}
\usepackage{mathrsfs}
\usepackage[top=1.5in, bottom=1.5in, left=1.25in, right=1.25in]	{geometry}
\usepackage{lmodern}

\usepackage{graphics}
\usepackage{euscript}
\usepackage{xcolor}

\usepackage{mathtools}
\mathtoolsset{showonlyrefs,showmanualtags}

\usepackage{hyperref} 
\hypersetup{
	colorlinks=true,       
	linkcolor=blue,          
	citecolor=magenta,        
	filecolor=magenta,      
	urlcolor=cyan           
}

\usepackage[msc-links]{amsrefs}

\newtheorem{theorem}{Theorem}[section]

\newtheorem{proposition}{Proposition}[section]
\newtheorem{lemma}{Lemma}[section]
\newtheorem{corollary}{Corollary}[section]

\theoremstyle{definition}
\newtheorem{definition}{Definition}[section]
\theoremstyle{definition}

\newtheorem{procedure}{Procedure}
\newtheorem{remark}{Remark}[section]

\theoremstyle{remark}

\def\Ch{{\mathfrak {Ch}}}

\def\r{{\rm r}}
\def\BMO{{\rm BMO }}
\def\BO{{\rm BO}}

\def\OSC{{\rm OSC}}
\def\SUP{{\rm SUP}}
\def\INF{{\rm INF}}
\def\supp{{\rm supp\,}}

\def\esssup{{\rm esssup\, }}
\def\essinf{{\rm essinf\, }}

\def\Gen{{\mathfrak {Gen}}}
\def\ZU{\ensuremath{\mathbb U}}
\def\ZV{\ensuremath{\mathbb V}}
\def\ZS{\ensuremath{\mathcal S}}
\def\MM{\ensuremath{\mathcal M}}
\def\ZM{\ensuremath{\mathfrak M}}
\def\ZB{\ensuremath{\mathscr B}}
\def\zB{\ensuremath{\mathfrak B}}
\def\ZA{\ensuremath{\mathscr A}}

\def\ZZ{\ensuremath{\mathbb Z}}

\def\ZI{\ensuremath{\textbf 1}}

\def\ZN{\ensuremath{\mathbb N}}

\def\ZP{\ensuremath{\mathscr P}}

\def\ZK{\ensuremath{\mathcal K}}

\def\ZQ{\ensuremath{\mathcal Q}}
\def\ZR{\ensuremath{\mathbb R}}

\def\ZT{\ensuremath{\mathbb T}}
\def\pr{\ensuremath{\mathrm {pr}}}
\def\ZL{\ensuremath{\mathcal L}}

\def\ch{{\mathfrak {Ch}}}
\def\ZG{{\mathscr G\,}}
\def\ZC{{\mathscr {C}}}

\numberwithin{equation}{section}
\def\md#1#2\emd{\ifx0#1
	\begin{equation*} #2 \end{equation*}\fi  
	\ifx1#1\begin{equation}#2\end{equation}\fi   
	\ifx2#1\begin{align*}#2\end{align*}\fi   
	\ifx3#1\begin{align}#2\end{align}\fi    
	\ifx4#1\begin{gather*}#2\end{gather*}\fi  
	\ifx5#1\begin{gather}#2\end{gather}\fi   
	\ifx6#1\begin{multline*}#2\end{multline*}\fi  
	\ifx7#1\begin{multline}#2\end{multline}\fi  
	\ifx8#1\begin{multline*}\begin{split}#2\end{split}\end{multline*}\fi
	\ifx9#1\begin{multline}\begin{split}#2\end{split}\end{multline}\fi
}
\newcommand {\e }[1]{\eqref{#1}}
\newcommand {\lem }[1]{Lemma \ref{#1}}
\newcommand {\rem }[1]{Remark \ref{#1}}
\newcommand {\cor }[1]{Corollary \ref{#1}}
\newcommand {\pro }[1]{Proposition \ref{#1}}

\newcommand {\trm }[1]{Theorem \ref{#1}}

\newcommand {\df }[1]{Definition \ref{#1}}

\title[] {New estimates for bounded oscillation operators}
\author{Grigori A. Karagulyan}
\address{Institute of Mathematics of NAS of RA, Marshal Baghramian ave., 24/5, Yerevan, 0019, Armenia} 
\address{Faculty of Mathematics and Mechanics, Yerevan State
University, Alex Manoogian, 1, 0025, Yerevan, Armenia} 
\email{g.karagulyan@ysu.am}

\thanks{The work was supported by the Science Committee of RA, in the frames of the research project 21AG‐1A045}

\subjclass[2010]{42C05, 42C10, 42C25, 42A55}
\keywords{Calder\'on-Zygmund operator, weighted inequalities, sparse operators, maximal function, martingale transform, homogeneous spaces}
	
\begin{document}

\begin{abstract}
	Bounded oscillation ($\BO$) operators were recently introduced and studied in \cite{Kar3}. Those are operators running on abstract measure spaces equipped with a ball-basis. It was proved that many operators in harmonic analysis (Calder\'on-Zygmund operators and their maximally modulations, Carleson type operators, martingale transforms, Littlewood-Paley  square functions, maximal function, etc) are $\BO$ operators. Various properties of $\BO$ operators were studied in \cite{Kar1,Kar3}.  A multilinear version of $\BO$ operators recently was considered in \cite{Ming}, recovering series of results of papers \cite{Kar1,Kar3} in multilinear setting. A main line of these works are so called sparse dominations results, a subject of intensive study of recent years, which is closely related to sharp weighted norm and exponential decay estimates of various operators. First sparse domination results were proved for the Calder\'on-Zygmund operators in \cite{Ler,CoRe,Lac}, providing a simplified proof to the $A_2$-conjecture that was earlier solved in \cite{Hyt}. 
	
	  In this paper we will consider a generalized version of $\BO$ operators, involving new parameters in the definition, as well as considering the operators on vector-valued function spaces.  With this definition we will capture some more operators to the class of $\BO$ operators. We prove sparse domination and exponential decay estimates, with various applications in harmonic analysis operators. We provide a new simplified approach, separating certain set theoretic proposition, which become a basic tool in the proofs of the main results.  For a "narrowed" class  of $\BO$ operators we also obtain new type of sparse estimation, involving mean oscillation instead of integral averages in the definition of sparse operators. Among with new corollaries we recover also series of results obtained in recent years. In particular, we prove the boundedness of  maximally modulated Calder\'on-Zygmund operators on spaces $\BMO$. 
\end{abstract}

	\maketitle  
 \tableofcontents		
\section{Introduction}
\subsection{The definition of bounded oscillation operators} The goal of this paper is to study new properties of bounded oscillation ($\BO$) operators recently introduced in \cite{Kar3},\cite{Kar1}.  It was shown in \cite{Kar3} that various classical operators in harmonic analysis (Calder\'on-Zygmund operator on homogeneous metric measure spaces, maximally modulated Calder\'on-Zygmund operators, in particular Carleson type maximal operators, martingale transforms and maximal function) are $\BO$ operators. It was also proved that $\BO$ operators share many common properties of those operators. In this paper we will consider a generalized version of $\BO$ operators, using new parameters in their definition, as well as considering the operators on vector-valued function spaces.  With this definition we capture some more operators in the class of $\BO$ operators (Riesz potentials, some Littlewood-Paley square functions, fractional maximal function, vector-valued martingale transform).   

To define the $\BO$ operators we will need some preliminary definitions and notation. $\BO$ operators are defined on abstract measure spaces equipped with a ball-basis, which is a basic concept in this theory.  Recall the definition of ball-basis from \cite{Kar3}.
\begin{definition} Let $(X,\ZM, \mu)$ be a measure space with a $\sigma$-algebra $\ZM$ and a measure $\mu$. A family of measurable sets $\ZB$ is said to be a ball-basis if it satisfies the following conditions: 
	\begin{enumerate}
		\item[B1)] $0<\mu(B)<\infty$ for any ball $B\in\ZB$.
		\item[B2)] For any points $x,y\in X$ there exists a ball $B\ni x,y$.
		\item[B3)] If $E\in \ZM$, then for any $\varepsilon>0$ there exists a (finite or infinite) sequence of balls $B_k$, $k=1,2,\ldots$, such that $\mu(E\bigtriangleup \cup_k B_k)<\varepsilon$.
		\item[B4)] For any $B\in\ZB$ there is a ball $ [B]\in\ZB $ (called {\rm hull-ball} of $B$), satisfying the conditions
		\begin{align}
			&B^*=\bigcup_{A\in\ZB:\, \mu(A)\le 2\mu(B),\, A\cap B\neq\varnothing}A\subset  [B],\label{h12}\\
			&\qquad\qquad\mu\left([B]\right)\le \ZK\mu(B),\label{h13}
		\end{align}
		where $\ZK$ is a positive constant independent of $B$. 
	\end{enumerate}
\end{definition}
\begin{definition}
	We say that a ball basis $\ZB$ is doubling if there is a constant $\eta>1$ such that for any ball $B\in \ZB$, with $B^*\neq X$, one can find a $B'\in \ZB$, satisfying
	\begin{equation}\label{h73}
		B\subsetneq B',\quad \mu(B')\le\eta  \cdot \mu(B).
	\end{equation}
\end{definition}
Notice that properties B1)-B4) in the definition of ball-basis is a selection of basic properties of classical balls (or cubes) in $\ZR^n$, which are crucial in the study of singular operators on $\ZR^n$. Moreover, it was proved in \cite{Kar3} that the families of metric balls in measure spaces of homogeneous type form a ball-basis with the doubling condition. Other examples of doubling ball-basis are the families of dyadic cubes in $\ZR^n$ and those extensions in general measure spaces. Recall the definition of a martingale filtration that provide typical examples of non-doubling ball-bases. Let $(X,\mu)$ be a measure space and $\{\ZB_n:\, n\in\ZZ\}$ be collections of measurable sets (called filtration) such that
\begin{itemize}
	\item each $\ZB_n$ forms a finite or countable partition of $X$,
	\item each $A\in \ZB_n$ is a union of some sets of $\ZB_{n+1}$,
	\item the collection $\ZB=\cup_{n\in\ZZ}\ZB_n$ generates the $\sigma$-algebra $\ZM$,
	\item for any points $x,y\in X$ there is a set $A\in X$ such that $x,y\in A$.
\end{itemize}
One can easily check that $\ZB$ satisfies  the ball-basis conditions B1)-B4), where for any $A\in \ZB$ we can choose $[A]=A^*$, since clearly $A^*\in \ZB$. We call a ball-basis, satisfying above relations, a ball-basis of martingale filtration. Observe that a martingale ball-basis is doubling if and only if $\mu(\pr(A) )\le c\mu(A)$ for any $A\in \ZB$ and for some constant $c>0$, where $\pr(A)$ denotes the parent-ball of $A$.

\begin{definition}
	A ball-basis $\ZB$ is said to be separable if the union of an arbitrary collection of balls (not necessarily countable) is measurable.
\end{definition}
Note that the separability is an important property, which enables the measurability of certain maximal type operators. It is well known that classical examples of ball bases are separable. Observe that for general ball bases such a property can fail. This gap can be covered either assuming the separability of the ball-basis or introducing outer measure and $L^p$ spaces of non-measurable functions on $X$ as it was done in \cite{Kar3}.  Dislike to \cite{Kar3}, in this paper we choose the first way, since we will be working in vector-valued $L^p$ spaces. Hence, in the sequel we will always suppose that our ball bases is separable, in particular, the sets $B^*$ defined in \e{h12} are measurable.

Let $\ZU$ and $\ZV$ be Banach spaces. Denote by $L^0(X,\ZV)$ the space of $\ZV$-valued measurable functions on $X$, and let $L^r(X,\ZU)$ be the $L^r$ space of $\ZU$-valued functions on $X$. An operator 
\begin{equation}\label{y58}
	T:L^r(X,\ZU)\to L^0(X,\ZV) 
\end{equation}
is said to be sublinear if for any functions $f,g\in L^r(X,\ZU)$ we have 
\begin{equation}
	\left\|T\left(f+g\right)\right\|_\ZV\le 	\left\|T\left(f\right)\right\|_\ZV+	\left\|T\left(g\right)\right\|_\ZV.
\end{equation}
Let the parameters $r\ge 0$ and $ \varrho\ge \rho>0$ be fixed. For a function $f\in L^r(X,\ZU)$ and a ball $B\in \ZB$ we set 
\begin{align}
	&	\langle f\rangle_B=\frac{1}{(\mu(B))^\rho}\left(\int_B\|f\|_\ZU^r\right)^\varrho, \quad \langle f\rangle_B^*=\sup_{A\in \ZB:A\supset B}\langle f\rangle_{A},\label{y56}\\
	&\OSC_B(f)=\esssup_{x,x'\in B}\|f(x)-f(x')\|_\ZU,\label{y97}\\
	&\SUP_B(f)=\esssup_{x\in B}\|f(x)\|_\ZU,\quad \INF_B(f)=\essinf_{x\in B}\|f(x)\|_\ZU.\label{y98}
\end{align}
\begin{definition}
	We say that sublinear operator $T:L^r(X,\ZU)\to L^0(X,\ZV)$ is a bounded oscillation operator if it satisfies the conditions
	\begin{enumerate}
		\item [T0)] (Weak type inequality) If $B\in \ZB$, then
		\begin{equation}\label{y57}
			\mu\{x\in B:\, \|T(f\cdot \ZI_B)(x)\|_\ZV>\lambda \cdot \langle f\rangle_{B}\}\le\left(\frac{\ZL_0}{\lambda}\right)^{1/\rho}\mu(B),\quad \lambda>0,
		\end{equation}
		\item [T1)](Localization) for every $B\in \ZB$ we have
		\begin{equation}\label{d7}
			\OSC_B\big(T(f\cdot \ZI_{X\setminus B^*})\big) \le\ZL_1\cdot \langle f \rangle^*_{B},
		\end{equation}
		\item [T2)](Connectivity) for any ball $A\in \ZB$ ($A^*\neq X$) there exists a ball $B\supsetneq A$ (i.e. $B\supset A$, $B\neq A$) such that
		\begin{equation}\label{d20}
			\SUP_A\big(T(f\cdot \ZI_{B^*\setminus A^*})\big)\le\ZL_2\cdot \langle f\rangle_{B^*},
		\end{equation}
	\end{enumerate}
whenever $f\in L^r(X,\ZU)$. Here $\ZL_0=\ZL_0(T)$, $\ZL_1=\ZL_1(T)$ and $\ZL_2=\ZL_2(T)$ are constants depending only on the operator $T$. 
\end{definition}
Choosing $ \varrho=\rho=1/r$ in the notations \e{y56} and considering real-valued function spaces, we obtain the original definition of $\BO$ operators studied in \cite{Kar3,Kar1}. Besides, in \cite{Kar3} weak type condition T0) was not included in the definition of $\BO$ operators, it was stated separately as an additional property of $T$. Since the weak type inequality is always used in the proofs of basic properties of $\BO$ operators, we find convenient to include it directly in the definition. Let us recall several initial properties of the original $\BO$ operators proved in \cite{Kar3}.  
\begin{itemize}
	\item If the ball-basis $\ZB$ is doubling and a sublinear operator $T$ satisfies T0) and T1) conditions, then it satisfies also T2) and so is a $\BO$ operator. 
	\item If $T$ is a $\BO$ operator, then the truncated operator defined by
\begin{equation}\label{y50}
	T^*f(x)=\sup_{B\in \ZB,\, x\in B}\left\|T\left(f\cdot\ZI_{X\setminus B^*} \right)\right\|_\ZV
\end{equation}
is also a $\BO$ operator, acting from $L^r(X,\ZU)$ into $L^r(X,\ZR)$.
\item If $\{T_\alpha\}$ is a family of $\BO$ operators with uniformly bounded constants $\ZL_0(T_\alpha)$, $\ZL_1(T_\alpha)$ and $\ZL_2(T_\alpha)$, then the maximally modulated operator 
\begin{equation}
	T(f)=\sup_\alpha\|T_\alpha(f)\|_\ZV
\end{equation}
is also a $\BO$ operator, acting from $L^r(X,\ZU)$ into $L^r(X,\ZR)$.
\item It was proved in \cite{Kar1} that $\BO$ operators based on doubling ball-basis satisfy classical good-$\lambda$ inequalities common for Calder\'on–Zygmund operators and martingale transforms.
\end{itemize}
\subsection{Sparse domination by fractional means}
A classical problem in harmonic analysis is the control of an operator by means of simple operators. A model example of such a phenomenon is the technique of domination by so-called sparse operators, which became a powerful tool to obtain sharp weighted norm inequalities in the series of recent works. The sparse operators was first explicitly appeared in Lerner’s work \cite{Ler} towards an alternative proof of the $A_2$-theorem due to Hyt\"onen \cite{Hyt}. Namely, it was obtained in \cite{Ler} a norm estimation of classical Calderón–Zygmund operators by sparse operators. This was later improved to a pointwise estimate independently by Conde-Alonso and Rey \cite{CoRe} and by Lerner and Nazarov \cite{LeNa}. Afterwards, Lacey \cite{Lac} obtained the same pointwise domination for general $\omega$-Calder\'{o}n-Zygmund operators with  $\omega$, satisfying the Dini condition.  Hyt\"onen, Roncal and Tapiola \cite{HyRo} elaborated the proof of Lacey \cite{Lac} to get a precise linear dependence of the domination constant on the characteristic numbers of the operator.  Lacey \cite{Lac} obtained also a pointwise sparse domination for general martingale transforms (including non-doubling cases of filtrations), giving a short and elementary proof of the $A_2$-theorem for martingale transforms due to Thiele, Treil and Volberg in \cite {TTV}. This result of Lacey was the first sparse domination results in a non-doubling context.  Karagulyan \cite{Kar3} obtained a pointwise sparse domination for the original $\BO$ operators with parameters $\varrho =\rho=1/r$, unifying the mentioned sparse estimation results in one. The result of \cite{Kar3} also enabled to obtain the sparse estimate of $\omega$-Calder\'{o}n-Zygmund operators in metric measure spaces of homogeneous type as well as for the maximally modulated Calder\'{o}n-Zygmund operators. Later, Lorist \cite{Lor} proved a sparse domination result for $l^p$-linear vector-valued $\BO$-operators based on the Lebesgue measure $\ZR^n$. Applying an extending approach, the main result of \cite{Lor}  then have been reproved in \cite{LeLo}. In fact, the results of \cite{Lor, LeLo} were stated in a slight different form, namely in the terms of so-called grand maximal operator without mentioning $\BO$-operators at all. 

Finally note that multilinear versions of many results of papers \cite{Kar3}, \cite{Kar1} recently were proved in \cite{Ming}. 

In this paper we extend the sparse domination result of \cite{Kar3} for vector-valued $\BO$ operators with arbitrary parameters $r>0$, $ \varrho\ge \rho>0$ (see \trm{T7}). We use the same technique as in \cite{Kar3}, separating the main ingredient of the proof. Namely, we proof a new function-free ball sparsification  result (see \pro{P11}), which easily implies a sparse domination for general $\BO$ operators.  \pro{P11} will be also used in a new type of sparse domination, involving the mean oscillation over balls (see \trm{T8}). 

Recall the definitions of sparse and martingale collections of balls.
\begin{definition}
	A collection of measurable sets (or balls) $\ZS$ is said to be sparse or $\gamma$-sparse if for any $B\in \ZS$ there is a set $E_B\subset B$ such that $\mu(E_B)\ge \gamma\mu(B)$ and the sets $\{E_B:\, B\in \ZS\}$ are pairwise disjoint, where  $0<\gamma<1$ is a constant.   
\end{definition}
\begin{definition}
	A family of sets $\ZS$ is said to be martingale type if every couple of elements $A,B\in \ZA$ satisfies one of these three relations $A\cap B=\varnothing$, $A\subset B$ or $B\subset A$.
\end{definition}
	\begin{definition}
		For two collection of sets $\ZS$ and $\bar \ZS$ we write $\bar \ZS \Subset\ZS$ if there is a one to one mapping from $\ZS$ to $\bar \ZS$ taking every $A\in \ZS$ to a $\bar A\in \ZS$ such that $\bar A\subset A$. 
	\end{definition}Further, positive constants depending on $\ZK$, $\eta$, $r$, $\varrho$ and $\rho$ (see \e {h13}, \e{h73}, \e{y56}, \e{y57}) will be called admissible constants and the relation $a\lesssim b$ will stand for $a\le c\cdot b$, where $c>0$  is admissible. We write $a\sim b$ if the relations  $a\lesssim b$ and $b\lesssim a$ hold simultaneously.
	\begin{theorem}\label{T7}
	Let $(X,\mu)$ be a measure space equipped with a ball-basis $\ZB$ and $T$ be a $\BO$-operator from $L^r(X,\ZU)$ to $L^0(X,\ZV)$. Then for any function $f\in L^r(X,\ZU)$ with $\supp f\subset B\in\ZB$ there exist a ball $B'\supset B$, as well as a $\gamma$-sparse family of balls $\ZA$ and a martingale-family of measurable sets $\bar \ZA\Subset \ZA$ such that
	\begin{align}
		& \bigcup_{G\in \ZA}G\subset B',\quad \mu(B')\lesssim \mu(B),\label{y52}\\
		&\mu\big\{x\in X:\,\sum_{G\in \ZA}\ZI_{\bar G}(x)>\lambda\big\}\lesssim \exp(-c\lambda)\mu(B),\quad \lambda>0,\label{y53}\\
		&\|Tf(x)\|_\ZV\lesssim C(T)\sum_{G\in \ZA}\left\langle f\right\rangle_{G}\ZI_{\bar G}(x),\text { a.e. } x\in B,\label{y51}
	\end{align}
	where $0<\gamma<1$, $c>0$ are admissible constants and $C(T)=\ZL_0(T)+\ZL_1(T)+\ZL_2(T)$. The same inequality holds also for the truncated operator $T^*$ from \e{y50}.
\end{theorem}
\begin{remark}
	We could omit conditions \e{y52} and \e{y53} in the statement of the theorem and write $\ZI_G$ instead of $\ZI_{\bar G}$ in \e{y51} likewise the sparse domination of \cite{Kar3}. With additional conditions  \e{y52} and \e{y53} we may capture some exponential decay inequalities proved first for classical Calderón–Zygmund operators in \cite{Kar2} a long time before the sparse domination results appeared, as well as for the $\BO$ operators proved in a recent work \cite{Kar1}.  Note that the exponential estimate of \cite{Kar1} is for $\BO$ operators with respect to doubling ball-bases. Applying \trm{T7} we extend this result on abstract measure spaces with general ball-bases. 
\end{remark}

\subsection{Sparse domination by mean oscillations}
For certain subclass of $\BO$ operators and their maximally modulations we prove a new type of sparse domination and exponential decay inequalities, involving mean oscillation instead of fractional means.   Using these results we obtain boundedness of maximally modulated Calder\'{o}n-Zygmund and some other operators on spaces $\BMO$. Note that dominations by fractional mean sparse operators of type \e{y51} can not be applied in $\BMO$-norm inequalities, since $\BMO$ spaces are not Banach lattices. 

Hence, we will consider a measure space $(X,\mu)$ and $\BO$ operators $T:L^r(X,\ZU)\to L^0(X,\ZV)$ under the following restrictions
\begin{enumerate}
	\item [R1)]ball-basis $\ZB$ is doubling,
	\item [R2)] $T$ is linear,
	\item [R3)] $T$ satisfies T0)-condition with parameters $\varrho=\rho=1/r$,
	\item [R4)] there exists a constant $\ZL_1(T)$ such that for any ball $B\in \ZB$ we have 
	\begin{equation}\label{y83}
		\OSC_B\big(T (f\cdot \ZI_{X\setminus B^*})\big) \le \ZL_1(T)\sup_{A\in \ZB:A\supset B}\log^{-1}\left(1+\frac{\mu(A)}{\mu(B)}\right)\langle f\rangle_{A}
	\end{equation}
	where $\langle f\rangle_A $ in \e{y57} and \e{y83} are considered with parameters $\varrho=\rho=1/r$, i.e.
	\begin{equation}
		\langle f\rangle_A =\left(\frac{1}{\mu(A)}\int_A\|f\|_\ZU^r\right)^{1/r},
	\end{equation}
	\item [R5)] for any ball $B\in \ZB$ and a number $\varepsilon>0$ there exists a ball $B'\supset B$ such that for any ball $B''\supset B'$ it holds the inequality $\OSC_B\big(T (\ZI_{B''})\big)\le \varepsilon$.
\end{enumerate}
\begin{definition}
	We say a BO operator is of restricted type if it satisfies R1)-R5) (in particular the ball-basis is doubling). A family of restricted type operators $\{T_\alpha\}$ is said to be uniformly restricted if conditions R3)-R5) are uniformly satisfied.
\end{definition}
\begin{remark}\label{R4}
	Notice that \e{y83} is a stronger version of T1)-condition, because of the additional logarithmic factor on the right hand side of \e{y83}. 
\end{remark}
We prove a sparse domination for the maximally modulations of operators satisfying R1)-R5). To state the result, for $f\in L^r(X,\ZU)$ and a ball $B\in \ZB$ define the mean oscillation analogues of quantities \e{y56} corresponding to parameters $\varrho=\rho=1/r$, setting 
\begin{align}	
	&f_B=\frac{1}{\mu(B)}\int_Bf,\\
	&\langle f\rangle_{\#,B}=\left(\frac{1}{\mu(B)}\int_B\|f-f_B\|_\ZU^r\right)^{1/r},\\
	& \langle f\rangle_{\#, B}^*=\sup_{A\in \ZB:A\supset B}\langle f\rangle_{\#,A}.
\end{align}
\begin{theorem}\label{T8}
Let a ball-basis $\ZB$ be doubling, $\{T_\alpha\}$ be a family of uniformly restricted $\BO$ operators and let $Tf=\sup_\alpha\|T_\alpha f\|_\ZV$. If
\begin{equation}\label{y84}
	C(T)=\ZL_0(T)+\sup_\alpha\ZL_1(T_\alpha)<\infty,
\end{equation}
then for any $f\in L^r(X,\ZU)$ and a ball $B$  there exist a ball $B'\supset B$, a sparse-collection of balls $\ZA$ and a martingale-collection of measurable sets $\bar \ZA\Subset \ZA$ satisfying \e{y52}, \e{y53} such that
	\begin{equation}\label{y94}
	|Tf(x)-m(B)|\lesssim C(T)\sum_{G\in \ZA}\langle f\rangle_{\#,G}^*\ZI_{\bar G}(x), \quad x\in B,
\end{equation}		
where $m(B)=m_{T(f)}(B)$ is a median of the real valued function $T(f)$ on $B$ (see Definition \ref{D2}).
\end{theorem}
The proof of \trm{T8} is based on the following inequality, where the $\beta$-oscillation (see \e{y100} for the definition) of a $\BO$ operator is estimated by a truncated sharp function.
\begin{theorem}\label{T5}
	Let $1\le r<\infty$ and $\ZB$ be a doubling ball-basis. If a family of $\BO$ operators $\{T_\alpha\}$ satisfies R2)-R5) and the corresponding maximal operator $Tf=\sup_\alpha\|T_\alpha (f)\|_\ZV$ satisfies \e{y84}, then the bound
	\begin{equation}\label{y30}
		\OSC_{B,\beta}\,T( f)\lesssim \left(\ZL_0(T)(1-\beta)^{-1/r}+\sup_\alpha \ZL_1(T_\alpha)\right) \langle f\rangle^*_{\#,B}.
	\end{equation}
	holds for any $f\in L^r(X,\ZU)$, a ball $B\in \ZB$ and a number $0<\beta<1$.
\end{theorem}
	\begin{remark}
		If $\mu(X)<\infty$ and $T(\ZI_X)$ is a constant function, then $T$ trivially satisfies condition R5) with $B'=X$. We just note that $\mu(X)<\infty$ implies $X\in\ZB$ (see \lem{L0}). Thus one can state simplified versions of Theorems \ref{T8} and \ref{T5} if $\mu (X)<\infty$.  Namely, we can suppose  $\{T_\alpha\}$ satisfies the restrictions R2)-R4) uniformly, as well as $T_\alpha(\ZI_X)\equiv c_\alpha$, $\|c_\alpha\|_\ZV\le \ZL_1(T_\alpha)$ instead of R5).
	\end{remark}
\begin{remark}
	The statement of \trm{T5} and so \trm{T8} is true also when $T$ is simply a restricted type $\BO$ operator. One just need replace the right hand side of \e{y94} by 	$\|Tf(x)-m(B)\|_\ZV$.
\end{remark}
Applying \e{y94}, we obtain new type of exponential inequality (\cor{C5}), which in turn implies the boundedness of restricted type $\BO$ operators and their maximally modulations on spaces $\BMO$ (\cor{C4}). The fact that ordinary Calderón–Zyg\-mund operators boundedly map $L^\infty$ to $\BMO$ was independently obtained by Peetre \cite{Pee}, Spanne \cite{Spa} and Stein \cite {Ste}. Peetre \cite{Pee} also observed that translation-invariant Calderón–Zygmund operators actually map $\BMO$ to itself. For a wider subclass of $\BO$ operators it can be easily proved the boundedness from $L^\infty $ to $\BMO$. 

\subsection{General maximal functions on $\BMO$ }
The technique used in the proof of \e{y94} and related corollaries can be applied also to maximal functions of general type. Here we use notations \e{y56} with parameters $\varrho=\rho=1/r$, $r\ge 1$ and suppose our ball-basis is doubling.
Define the volume distance between a point $x\in X$ and a ball $B$ by
\begin{equation}
	d(x,B)=\inf_{A\in \ZB:\, A\supset B\cup \{x\}}\mu(A).
\end{equation}
\begin{definition}
	Let $\omega:(0,1)\to(0,1)$ be a modulus of continuity, that is a monotone increasing and subadditive function. A family of non-negative real valued functions $\phi=\{\phi_B:\, B
	\in \ZB\}\subset L^1(X)$ is said to be $\omega$-regular if 
	\begin{enumerate}
		\item $\|\phi_B\|_{L^1(X)}=1$,
		\item there is an increasing function $\gamma:[1,\infty)\to [1,\infty)$ such that
		\begin{equation}
			\phi_B(x)\le \gamma\left(\frac{\mu(A)}{\mu(B)}\right)\phi_A(x)
		\end{equation}
		for every two balls $B\subset A$.
		\item there are positive constants $c_1,c_2$ such that for any ball $B$ we have the inequalities
		\begin{equation}\label{y4}
			c_1\cdot \frac{\ZI_B(x)}{\mu(B)}\le\phi_B(x)\le c_2\cdot \omega\left(\frac{\mu(B)}{d(x,B)}\right)\frac{1}{d(x,B)}, \quad x\in X .
		\end{equation}
	\end{enumerate}
\end{definition}
\begin{definition}
	A family of balls $\ZG(x)\subset \ZB$ is said to be complete at a point $x\in X$ if 
	\begin{enumerate}
		\item $x\in B$ for any $B\in \ZG(x)$, 
		\item  for any ball $A\ni x$ (not necessarily from $\ZG(x)$) there is a ball $\bar A\in \ZG(x)$ with $\bar A\supset A,\quad \mu(\bar A)\le \eta \mu(A)$,
	\end{enumerate}
	where $\eta \ge 1$ is a constant.  
\end{definition}
Let $\ZG=\cup_{x\in X}\ZG(x)$, where each $\ZG(x)$ is a family of balls complete at $x$ and let $\phi=\{\phi_B:\, B\in \ZB\}$ be a $\omega$-regular family of functions on $X$.  For $f\in L^r(X,\ZU)$ and a measurable set $B$ we define 
\begin{equation}
	\langle f\rangle_{\phi_B}=\int_X\|f\|_\ZU\phi_B,
\end{equation}
and consider the maximal function
\begin{equation}\label{y6}
	\MM^{\phi,\ZG} f(x)=\sup_{B\in \ZB(x)}\langle f\rangle_{\phi_B},
\end{equation}
which is an operator acting from $L^r(X,\ZU)$ to $L^0(X,\ZR)$.
\begin{theorem}\label{T2}
	Let $\ZB$ be a doubling ball-basis. If a modulus of continuity satisfies
	\begin{equation}\label{y7}
		\|\omega \|=1+\int_0^1\frac{\omega(t)\log(1/t)}{t}dt<\infty,
	\end{equation}
	then for any function $f\in L^r(X,\ZU)$ and a ball $B$
	\begin{equation}\label{2}
		\OSC_{B,\alpha}(\MM^{\phi,\ZG} (f))\lesssim \|\omega\|(1-\alpha)^{-1/r}\langle f\rangle^*_{\#,B},\quad B\in \ZB,
	\end{equation}
	for any $0<\alpha<1$.
\end{theorem}
Thus we can state a corollary analogous to \trm{T8}.
\begin{corollary}\label{KC1}
	If a ball-basis $\ZB$ is doubling, 	then for any $f\in L^r(X,\ZU)$ and a ball $B$  there exist a ball $B'\supset B$, a sparse-collection of balls $\ZA$ and a martingale-collection of measurable sets $\bar \ZA\Subset \ZA$ satisfying \e{y52}, \e{y53} such that
	\begin{equation}
	|\MM^{\phi,\ZG} f(x)-m(B)|\lesssim \|\omega\|\sum_{G\in \ZA}\langle f\rangle_{\#,G}^*\ZI_{\bar G}(x), \quad x\in B,
	\end{equation}		
	where $m(B)$ is a median of the real function $\MM^{\phi,\ZG} (f)$ on $B$ (see Definition \ref{D2}).
\end{corollary}
\begin{corollary}\label{KC2}
	Under the assumptions of \trm{T2} the maximal function $\MM^{\phi,\ZG}$ is bounded from $\BMO(X,\ZU)$ to $\BMO(X,\ZR)$. More precisely, 
	\begin{equation}
		\|\MM^{\phi,\ZG}(f)\|_{\BMO(X,\ZR)}\lesssim \|f\|_{\BMO(X,\ZU)}
	\end{equation}
	for any function
	\begin{equation}
		f\in L^r(X,\ZU)\cap \BMO(X,\ZU).
	\end{equation} 
\end{corollary}
\begin{remark}
	In the case when $\MM^{\phi,\ZG}$ is the centered maximal operator on $\ZR^d$, its boundedness on $\BMO(\ZR^d)$ is the classical result of Bennett-DeVore-Sharpley \cite{BDS}. To the best of our knowledge this was the only known maximal function that is bounded on $\BMO$. The suggested proof of \cite{BDS} quite different and in particular it can not be applied for the non-centered maximal function on $\ZR^d$.
\end{remark}
\section{Preliminary properties of ball bases}\label{S2}
Let $\ZB$ be a ball-basis in the measure space $(X,\mu)$. From B4) condition it follows that if balls $A,B$ satisfy $A\cap B\neq\varnothing$, $\mu(A)\le 2\mu(B)$, then $A\subset B^*$. This property will be called two balls relation. The notation $B^*$ in \e{h13} can be also used for arbitrary measurable set $B$. In particular, the notation $B^{**}$ will stand for the $(B^*)^*$. We say a set $E\subset X$ is bounded if $E\subset B$ for a ball $B\in\ZB$.
\begin{lemma}\label{L0}
		If $(X,\mu)$ is a measure space equipped with a ball-basis $\ZB$ and $\mu(X)<\infty$, then $X\in \ZB$.
\end{lemma}
\begin{proof}
Fix an arbitrary point $x_0\in X$ and let $\ZA\subset \ZB$ be the family of balls containing $x_0$. By B2) condition we can write $X=\cup_{A\in \ZA}A$. Then there exists a $B\in \ZA$ such that $\mu(B)>\frac{1}{2} \sup_{A\in \ZA}\mu(A)$ and using B4) condition we obtain
\begin{equation*}
X=\cup_{A\in \ZA}A\subset B^*\subset [B].
\end{equation*}
Thus $X=[B]\in \ZB$.
\end{proof}
\begin{lemma}\label{L5}
	Let $(X,\mu)$ be a measure space equipped with a ball-basis $\ZB$. Then there exists a sequence of balls $G_1\subset G_2\subset \ldots\subset G_n\subset \ldots$ such that $X=\cup_kG_k$. Moreover, for any ball $B$ there is a ball $G_n\supset B$.
\end{lemma}

\begin{proof}
	Fix an arbitrary $x_0\in X$ and let $\ZA\subset \ZB$ be the family of balls containing the point $x_0$. Choose a sequence $\eta_n\nearrow\eta=\sup_{A\in \ZA}\mu(A)$, where $\eta$ can also be infinity. Let us see by induction that there is an increasing  sequence of balls $G_n\in \ZA$ such that $\mu(G_n)> \eta_n$ and $G_n^*\subset G_{n+1}$. The base of induction is trivial. Suppose we have already chosen the first balls $G_k$, $k=1,2,\ldots, l$. There is a ball $A\in \ZA$ so that $\mu(A)>\eta_{l+1}$. Let $A_l$ be the biggest in measure among two balls $A$ and $G_l$. By property B4) we have $A\cup G_l\subset A_l^*$ and there exists a ball $G_{l+1}\supset A_l^{**}$. This implies $\mu(G_{l+1})\ge \mu(A)>\eta_{l+1}$ and $G_{l+1}\supset G_l^*$, completing the induction.
	Let us see that $G_n$ is our desired sequence of balls. Indeed, let $B$ be an arbitrary ball. By B2) property there is a ball $A$ containing $x_0$, such that $A\cap B\neq \varnothing$. Then by property B4) we may find a ball $C\supset A\cup B$. For some $n$ we will have $\mu(C)\le 2\mu(G_n)$ and so once again using B4), we get $B\subset  C\subset G_{n}^*\subset G_{n+1}$. 
\end{proof}

\begin{lemma}\label{L1-1}
	Let $(X,\mu)$ be a measure space with a ball bases $\ZB$. If $E\subset X$ is bounded and a family of balls $\ZG $ is a covering of $E$, i.e. $E\subset \bigcup_{G\in \ZG}G$, then there exists a finite or infinite sequence of pairwise disjoint balls $G_k\in \ZG$ such that $E \subset \bigcup_k G_k^{*}$.
\end{lemma}
\begin{proof}
	The boundedness of $E$ implies $E\subset B$ for some $B\in \ZB$. If there is a ball $G\in\ZG$ so that $G\cap B\neq \varnothing$, $\mu(G)> \mu(B)$, then by two balls relation we will have $E\subset B\subset G^{*}$. Thus our desired sequence can be formed by a single element $G$. Hence we can suppose that every $G\in\ZG$ satisfies $G\cap B\neq \varnothing$, $\mu(G)\le \mu(B)$ and again by two balls relation $G\subset B^*$. Therefore, $\bigcup_{G\in\ZG}G\subset B^{*}$.
	Choose $G_1\in \ZG$, satisfying $\mu(G_1)> \frac{1}{2}\sup_{G\in\ZG}\mu(G)$. Then, suppose by induction we have already chosen elements $G_1,\ldots,G_k$ from $\ZG$. Choose $G_{k+1}\in \ZG$  disjoint with the balls $G_1,\ldots,G_k$ such that
	\begin{equation}\label{b1}
		\mu(G_{k+1})> \frac{1}{2}\sup_{G\in \ZG:\,G\cap G_j=\varnothing,\,j=1,\ldots,k}\mu(G).
	\end{equation}
	If for some $n$ we will not be able to determine $G_{n+1}$ the process will stop and we will get a finite sequence $G_1,G_2,\ldots, G_n$. Otherwise our sequence will be infinite. We shall consider the infinite case of the sequence (the finite case can be done similarly). Since the balls $G_n$ are pairwise disjoint and $G_n\subset B^{*}$, we have $\mu(G_n)\to 0$. Choose an arbitrary $G\in\ZG$ with $G\neq G_k$, $k=1,2,\ldots $ and let $m$ be the smallest integer satisfying
	$\mu(G)\ge 2\mu(G_{m+1})$.
	So we have $G \cap G_j\neq\varnothing$
	for some $1\le j\le m$, since otherwise by \e{b1}, $G$ had to be chosen instead of $G_{m+1}$. Besides, we have $\mu(G)< 2\mu(G_{j})$ because of the minimality property of $m$, and so by two balls relation $G\subset G_{j}^{*}$. Since $G\in \ZG$ was chosen arbitrarily, we get $E \subset\bigcup_{G\in \ZG} G\subset \bigcup_k G_k^{*}$.
\end{proof}

\begin{definition}\label{D1}
	For a measurable set $E\subset X$ a point $x\in E$ is said to be density point if for any $\varepsilon>0$ there exists a ball $B\ni x$ such that
	\begin{equation*}
		\mu(B\cap E)>(1-\varepsilon )\mu(B).
	\end{equation*} 
	We say a ball basis satisfies the density property if for any measurable set $E$ almost all points $x\in E$ are density points. 
\end{definition}
\begin{lemma}\label{L12}
	Every ball basis $\ZB$ satisfies the density condition.
\end{lemma}
\begin{proof}
Applying \lem{L5} one can check that it is enough to establish the density property for the bounded measurable sets. Suppose to the contrary there exist a bounded measurable set $E\subset X$ together its subset $F\subset E$ (maybe non-measurable) with an outer measure
\begin{equation}
	\mu^*(F)=\inf_{A\supset F,\,A\in \ZM}\mu(A)>0
\end{equation}
 such that
	\begin{equation}\label{z38}
		\mu(B\setminus E)>\alpha\mu(B) \text{ whenever }B\in\ZB,\, B\cap F\neq\varnothing,
	\end{equation}
where $0<\alpha<1$. According to the definition of the outer measure one can find a measurable set $\bar F$ such that 
	\begin{equation}\label{z2}
		F\subset \bar F\subset E, \quad \mu(\bar F)<\mu^*(F) +\varepsilon.
	\end{equation}
	By B3)-condition there is a sequence of balls $B_k$, $k=1,2,\ldots $, such that
	\begin{equation}\label{z42}	
		\mu\left(\bar F\bigtriangleup (\cup_k B_k)\right)<\varepsilon.
	\end{equation}
	We discuss two kind of balls $B_k$, satisfying either $B_k\cap F=\varnothing$ or $B_k\cap F\neq\varnothing$. For the first collection we have 
	\begin{equation*}
		\mu^*(F)\le \mu\left(\bar F\setminus \bigcup_{k:\,B_k\cap F=\varnothing} B_k\right)\le \mu(\bar F)<\mu^*(F)+\varepsilon,
	\end{equation*}
	which implies 
	\begin{equation*}
		\mu\left(\bar F\bigcap \left(\bigcup_{k:\,B_k\cap F=\varnothing} B_k\right)\right)=\mu(\bar F)- \mu\left(\bar F\setminus \bigcup_{k:\,B_k\cap F=\varnothing} B_k\right)<\varepsilon.
	\end{equation*}
	Thus, combining also \e{z42} we obtain
	\begin{equation}\label{z3}
		\mu\left(\bigcup_{k:\,B_k\cap F=\varnothing} B_k\right)<2\varepsilon.
	\end{equation}
For the balls $B_k\cap F\neq\varnothing$, using \e {z38} and \e{z2} we have
	\begin{equation}\label{z41}
		\mu(B_k\setminus \bar F)\ge \mu(B_k\setminus E)>\alpha \mu(B_k),\quad k=1,2,\ldots.
	\end{equation}
		Since $E$ and so $\bar F$ are bounded, applying \lem {L1-1} and \e{z42}, one can find a subsequence of pairwise disjoint balls $\tilde B_k$, $k=1,2,\ldots$, such that
	\begin{equation*}
		\mu\left(\bar F\setminus\cup_k\tilde B_k^{*}\right)<\varepsilon.
	\end{equation*} 
	Thus, from B4)-condition, \e{z42}, \e {z3} and \e {z41}, we obtain
	\begin{align*}
		\mu^*(F)< \mu(\bar F)&\le \mu\left(\cup_k\tilde B_k^{*}\right)+\varepsilon\le \ZK\sum_k\mu(\tilde B_k)+\varepsilon\\
		&= \ZK\mu\left(\bigcup_{k:\,\tilde B_k\cap F=\varnothing} \tilde B_k\right)+\ZK\sum_{k:\,\tilde B_k\cap F\neq\varnothing} \mu(\tilde B_k)+\varepsilon\\
		&<2\ZK\varepsilon+\frac{\ZK}{\alpha}\sum_k\mu(\tilde B_k\setminus \bar F)+\varepsilon\\
		&\le (2\ZK+1)\varepsilon+\frac{\ZK}{\alpha}\mu\left(\bar F\bigtriangleup( \cup_k B_k)\right)<\varepsilon\left(2\ZK+1+\frac{\ZK}{\alpha}\right).
	\end{align*}
	Since $\varepsilon $ can be arbitrarily small, we get $\mu(F)=0$ and so a contradiction.
\end{proof}
Denote by $\# E$ the number of elements of a finite set $E$.
\begin{lemma}\label{L1-2}  Let $A\in \ZB$ and $\ZG$ be a family of pairwise disjoint balls such that each $G\in \ZG$ satisfies the relations 
	\begin{align}
		&G^*\cap A \neq \varnothing,\label{y36}\\
		& 0<c_1\le \mu(G)\le c_2.\label{a41}
	\end{align}
	Then $\ZG$ is finite and for the number of its elements we have $\#\ZG\lesssim\max\{c_2, \mu(A)\}/c_1$.
\end{lemma}
\begin{proof} Suppose $G_1,G_2,\ldots, G_N$ are some elements of $\ZG$ with hull-balls satisfying $\mu([G_1])\ge \max_j\mu([G_i])$. Consider the ball
	\begin{align}\label {a42}
		B=\left\{
		\begin{array}{lc}
			A&\hbox{ if } \mu(A)>\mu([G_1]),\\
			\left[G_1\right]&\hbox { if } \mu(A)\le \mu([G_1]).\\
		\end{array}
		\right.
	\end{align}
	Using \e{y36}, one can check that for both cases in \e{a42} we have $G_k\subset [G_k]\subset B^{**}$, $k=1,\ldots, N$.	Thus, since $G_k$ are pairwise disjoint, from \e {a41} we obtain
	\begin{align}
		N\cdot c_1&\le \mu\left(\cup_{k} G_k\right)\le \mu(B^{**})\\
		&\le\max\{\mu(A^{**}),\mu([G_1]^{**})\}\lesssim \max\{c_2, \mu(A)\},
	\end{align}
	which completes the proof of lemma.
\end{proof}
\begin{remark}\label{R2}
	One can prove a similar lemma with $G^{**}\cap A \neq \varnothing$ instead of \e {y36}.
\end{remark}
\begin{lemma}\label{L2}
	If $\ZB$ is a doubling ball-basis, then for any ball $B\in\ZB$ with $B^*\neq X$ there is a ball $B'$, satisfying 
	\begin{equation}\label{x42}
		B\subsetneq B',\quad 2\mu(B)\le \mu( B')\le\eta\ZK \mu(B).	
	\end{equation}
\end{lemma}
\begin{proof}
	Let $B\in \ZB$, $B^{*}\neq X$. Suppose to the contrary there is no ball $B'$ satisfying \e {x42} and denote
	\begin{equation*}
		\ZA=\{A\in\ZB:\,B\subseteq A,\, \mu(A)\le \eta\ZK\mu(B)\}.
	\end{equation*}
	According to the assumption every $A\in\ZA$ should satisfy the bound $\mu(A)<2\mu(B)$. This and B4) imply
$\bigcup_{A\in\ZA}A\subset B^*\subset [B]\in \ZB$, where $\mu([B])\le \ZK\mu(B)$. Thus, we have $[B]\in \ZA$ and so 
$\bigcup_{A\in\ZA}A= B^*=[B]$. On the other hand, according to the doubling condition, there is a ball 
$A'\supsetneq [B]$
	such that $\mu(A')\le\eta\cdot \mu([B])\le \eta\ZK\mu(B)$. The latter implies $A'\in\ZA$ and therefore $A'\subset [B]$ that gives a contradiction. Lemma is proved.
\end{proof}
The following lemma is in the same spirit as \lem{L5} above.
\begin{lemma}\label{L1}
	Let $\ZB$ be a doubling ball basis. Then for any balls $A\subset B$ there exists a sequence of balls $A=A_0\subset A_1\subset \ldots\subset A_n=[B]$ such that $\mu(A_{k+1})\lesssim \mu(A_k)$.
\end{lemma}
\begin{proof}
	The case of $\mu(B)\le 2\mu([A])$ is trivial, we may simply choose $A_0=A$ and $A_1=B$. So we can suppose $\mu(B)> 2\mu([A])$. Applying \lem{L2} we may find a sequence of balls $A=B_0, B_1,\ldots, B_{n-1}$ such that $[B_k]\subset B_{k+1}$ and $2\mu([B_k])\le \mu(B_{k+1})\le \eta \ZK \mu([B_k])$, $\mu(B)\le\mu(B_{n-1})\lesssim \mu(B)$. Obviously, we may define our desired sequence by $A_0=A$, $A_n=[B]$ and $A_k=[B_k]$, $k=1,2,\ldots,n-1$.
\end{proof}

\section{A set theoretic propositions}\label{S3}
For two measurable sets $E$ and $F$ the notation $E\subset_{a.e} F$ will stand for the relation $\mu(E\setminus F)=0$.
\begin{lemma}\label{L8}
	Let $F$ and $E\subset F$ be measurable sets, where $E$ is bounded as well. Then there exists a countable family of balls $\ZG=\ZG(E)$ such that
	\begin{align}
		&E\subset_{a.e.} \bigcup_{G\in \ZG}G, \label{a14}\\
		&\sum_{G\in \ZG}\mu(G)\le 2\ZK\mu(F).\label{a27}
	\end{align}
Besides, for any ball $G'\supsetneq G \in \ZG$ we have
	\begin{equation}\label{aa15}
		\mu(G'\cap F)<\mu(G')/2.
	\end{equation}
	If $G$ is doubling, then \e{aa15} (with the constant $2\eta\ZK$ on the right-hand side) holds also for $G'=G$.
\end{lemma}
\begin{proof}
	First consider a general ball-basis. Let $D$ be the density points set of $F$ (see \df{D1}). Thus for any $x\in D$ the family of balls
	\begin{equation*}
		\ZC_x=\{C\in\ZB:\, x\in C,\,\mu(C\cap F)\ge\mu(C)/2\}
	\end{equation*}
	is nonempty and therefore there is  a ball $B_x\in \ZC_x$ such that  $\mu(B_x)>\frac{1}{2}\sup_{C\in \ZC_x}\mu(C)$.
	From the two ball relation it follows that $\cup_{C\in \ZC_x}C\subset B_x^*$.
	Thus, if a ball $G'$ satisfies $G'\supsetneq B_x^*$, then $x\in G'\notin \ZC$ and so
	\begin{equation}\label{y8}
		\mu(G'\cap F)<\mu(G')/2.
	\end{equation}
	Since $\{B_x:\,x\in E\cap D\}$ is a ball-covering for $E\cap D$, applying \lem{L1-1}, we find sequence of pairwise disjoint balls $\{B_{x_k}\}$ such that 
	\begin{equation}\label{y5}
		E\cap D\subset \cup_kB_{x_k}^*.
	\end{equation}
	Choose a hull-ball $G_k\supset B_{x_k}^*$ such that $\mu(G_k)\le \ZK\mu(B_{x_k})$. We claim that the family $\ZG(E)=\{G_k\}$ satisfies the conditions of lemma. Indeed, \e{aa15} follows from the observation in \e{y8}.  By density \lem{L12} we have $\mu(F\setminus D)=0$ and combining also \e{y5}, we obtain \e{a14}. Then, since 
	$B_x\in \ZC_x$, it follows that
	\begin{align}
		\sum_k\mu(G_k)\le \ZK \sum_k\mu(B_{x_k})< 2\ZK \sum_k\mu(B_{x_k}\cap F)\le 2\ZK \mu(F)
	\end{align}
	that gives \e{a27}. This completes the proof in the case of non-doubling ball-basis. Now let $\ZB$ is doubling. Applying the above procedure first we get $\ZG(E)$, then replace each $G\in \ZG(E)$ by a ball $\tilde G\supsetneq G$ satisfying $\mu(\tilde G)\le \eta\mu(G)$. Clearly, all the conditions of \lem{L8} will be satisfied for the new collection as well, besides we will have \e{aa15} for $G'=G$.
\end{proof}
\begin{remark}\label{R3}
	Obviously we can also claim in the statement of \lem{L8} an extra condition 
	\begin{equation}\label{y9}
		E\cap G\neq\varnothing,\quad G\in \ZG(E).
	\end{equation}
	One just need to remove those balls from $\ZG(E)$, which do not satisfy \e{y9}.
\end{remark}

\begin{definition}
	A family of balls $\ZA$ is said to be a tree-collection with a root-ball $A_0\in \ZA$ if for each $A\in \ZA$ there is an attached children-collection $\ch(A)\subset \ZA$ (it can also be empty), $A\notin \ch(A)$, such that for every element $A\in \ZS$ there is a sequence of balls $\{A_0, A_1, \ldots, A_n=A\}\subset \ZA$ with $A_{j+1}\in \ch(A_j)$, $j=0,1,\ldots,n-1$. The relation $G\in \ch(A)$ otherwise will be denoted by $A=\pr(G)$, meaning that $A$ is the parent of $G$.  We say $G$ is in the generation of a ball $A\in \ZS$ if $G\in \ch^n(A)$ for some $n\ge 1$. The generation of a ball $A\in \ZA$ will be denoted by $\Gen(A)=\cup_{n\ge 1}\ch^n(A)$. By the definition for the root-ball we have $\Gen(A_0)=\ZA\setminus \{A_0\}$.
\end{definition}
The notation $n\ll m$ ($n\gg m$) for two integers $n,m$ denotes $n<m-1$ ($n> m+1$) and $n\asymp m$ will stand for the condition $|m-n|\le 1$.
\begin{proposition}\label{P11}
	Let $\{F_B:\, B\in \ZB\}$ be a family of measurable sets such that $\mu(F_B)<\alpha \mu(B)$, where $0<\alpha<1/10\ZK^7$. Then for any ball $A_0$ one can find a countable sparse-tree-collection of balls $\ZS$ containing a double-hull ball $[[A_0]]$ as a root-ball, such that 
	\begin{enumerate}
		\item [a1)] $G\in \ch(A)\,\Rightarrow \, G\subset  A$,
		\item [a2)] $A_0\subset_{a.e.} \cup_{A\in \ZS}A\setminus F_{A}$
		\item [a3)] for any $A\in \ZS$ it holds the bound
		\begin{equation}\label{s27}
			\sum_{G\in \ch(A)}\mu(G)\lesssim\alpha  \mu(A),
		\end{equation}
		\item [a4)] if $G\in \ch(A)$, then for any ball and $G'\supsetneq G$
		\begin{equation}\label{y27}
			\mu(G'\cap F_A)< \mu(G')/2\text { if }G\in \ch(A),
		\end{equation}
	\end{enumerate}
\end{proposition}	
If the ball basis $\ZB$ is doubling, then we can claim \e{y27} for $G'=G$ as well.
\begin{proof}
	First we construct a tree-collection of balls $\ZA$ with the root-ball $A_0$, then the desired sparse family $\ZS$ will be obtained applying certain removal procedures on $\ZA$. The elements of $\ZA$ will be determined inductively by an increasing order of generations levels. 
	First, we apply \lem{L8} for $F=F_{[[A_0]]}$ and $E=[A_0]\cap F_{[[A_0]]}$. We get a child-ball collection $\ch(A_0)=\ZG(E)$ satisfying conditions \e{a14}-\e{aa15}. Then we do the same with every $A\in \ch(A_0)$, again applying \lem{L8} for $F=F_{[[A]]}$ and $E=[A]\cap F_{[[A]]}$. With this we obtain the second generation $\ch^2(A_0)$. Continuing this procedure to infinity we will get generation families  $\ch^n(A_0)$, $n=0,1,2,\ldots$, which union we denote by $\ZA$. By the claims of \lem{L8} (see also \rem{R3}) for every $A\in \ZA$ we have
	\begin{align}
		&[A]\cap F_{[[A]]}\cap G\neq\varnothing,\quad G\in \ch(A), \label{s9}\\
		&[A]\cap F_{[[A]]}\subset_{a.e.} \bigcup_{G\in \ch(A)}G, \quad \label{s14}\\
		&\sum_{G\in \ch(A)}\mu(G)\le 2\ZK\mu(F_{[[A]]})\\
		&\qquad\qquad\quad\le 2\alpha\ZK  \mu([[A]])\le 2 \alpha\ZK^3\mu(A)< \mu(A)/5\ZK^4,\label{y72}\\
		&\mu(G'\cap F_{[[A]]})<\mu(G')/2,\label{s15}
	\end{align}
	where \e{s15} holds for any ball $G'\supset [G],\, G\in \ch(A)$. Besides, from \e{y72} we get $\mu(G^{**})\le \ZK^2\mu(G)<\mu(A)$ and so $G^{**}\subset A^{**}$, provided $G\in \ch(A)$. The latter implies 
	\begin{equation}\label{y11}
		G^{**}\subset A^{**}\text{ if }G\in \Gen (A).
	\end{equation}
	
	Now we apply removal procedures on $\ZA$, removing some elements of $\ZA$. As we will see below, 
	removing an element $A\in \ZA$, we also remove all the elements of its generation $\Gen (A)$. Thus relations \e{s9}, \e{y72} and \e{s15} keep holding during the entire process of reduction. 
	
	To start the description of the process, we let $R=\ZK^2$ and for $B\in \ZB$ denote $\r(B)=[\log_R \mu(B)]$. Since $G\in \ch(A)$ implies $\mu(G)< \mu(A)/R$ (see \e{y72} ), we can say
	\begin{equation}\label{y73}
		G\in \ch(A)\Rightarrow \r(G)<\r(A).
	\end{equation}
	Thus the collections of balls
	\begin{align*}
		\ZA_k&=\{B\in \ZA:\, \r(B)=k\}\\
		&=\left\{B\in \ZA:\, R^{k}\le \mu(B)<R^{k+1}\right\},\quad k\le k_0=\r(A_0),
	\end{align*}
	give a partition of $\ZA$, i.e. we have $\ZA=\cup_{k\le k_0} \ZA_k$, where $\ZA_{k_0}=\{A_0\}$. 
	The removal of the elements of $\ZA$ will be realized in different stages. The content of $\ZA_{k_0}$ will not be changed. In the $n$-th stage only the contents of the families $\ZA_k$ with $k\le k_0-n$ can be changed. Besides, at the end of the $n$-th stage $\ZA_{k_0-n}$ will be fixed and remain the same till the end of the process. Suppose by induction the $l$-th stage of the process has been already finished and so the families $\ZA_k$, $k=k_0,k_0-1,k_0-2,\ldots, k_0-l$ have already fixed. In the next $(l+1)$-th stage we will apply the following two procedures consecutively:
	\begin{procedure} Remove any element $G\in \ZA_{k_0-l-1}$ together with all the elements of his generation $\Gen(G)$ if there exists a $B\in \ZA$ satisfying the conditions  
		\begin{align}	
			&G^{**}\cap B\neq \varnothing ,\label{d29}\\
			&\r(\pr^k(G))\ll\r(B)\ll \r(\pr^{k+1}(G)),\label{d32}
		\end{align}
		for some integer $k\ge 0$ (for $\ll$ see the notation before \pro{P11}).
	\end{procedure}
	\begin{remark}
		Observe that if an element $G$ is removed because of a ball $B$ satisfying conditions \e {d29} and \e {d32} of Procedure 1, then we should have 
		\begin{equation}\label{y10}
			\r(G)\ll\r(B)\ll\r(\pr(G))
		\end{equation}
		that means at any time of the process \e {d32} can hold only with $k=0$. Indeed, the left hand side inequality in \e{y10} immediately follows from \e {d32}. To prove the right one, suppose to the contrary \e {d32} holds with $k\ge 1$. Thus, according to \e{y73}, we can write
		\begin{equation*}
			G'=\pr^k(G)\in \bigcup_{j=0}^{l}\ZA_{k_0-j}. 
		\end{equation*}
		Since $G'^{**}\supset G^{**}$ (see \e {y11}), we have $G'^{**}\cap B\neq \varnothing $. On the other hand  \e {d32} can be written by $\r(G')\ll\r(B)\ll\r(\pr(G'))$. We thus conclude that $G'$ satisfies the conditions of the Procedure 1, so $G'$  together with his generation $\Gen(G')$ (include $G$) had to be removed in one of the previous stages of the process. This is a contradiction and so $k=0$. 
		\begin{remark}
			We also observe that if some $G$ is removed because of \e{d29} and \e{d32}, then the ball $B$ in \e{d32} will never be removed from $\ZA$ during the entire process of induction.
		\end{remark}
	\end{remark}
	\begin{procedure}
		Apply \lem {L1-1} to the rest of the elements $\ZA_{k_0-l-1}$ having after Procedure 1. The application of \lem {L1-1} removes some more elements of $\ZA_{k_0-l-1}$. If an element $A$ is removed, then the generation $\Gen(A)$ will also be removed. 
	\end{procedure}
	\begin{remark}
		After the Procedure 2 the elements of $\ZA_{k_0-l-1}$ become pairwise disjoint. Besides, we will have   
		\begin{equation}\label{a70}
			\bigcup_{G\in \ZA_{k_0-l-1}(\text{\rm before Procedure 2})} G\subset \bigcup_{G\in \ZA_{k_0-l-1}(\text{\rm after Procedure  2})}G^{*}.
		\end{equation}
	\end{remark}
	After these two procedures the family $\ZA_{k_0-l-1}$ will be fixed. Hence, finishing the induction process, we get the final state of $\ZA$ which will be denoted by $\bar\ZA$. Since after Procedure 2 in the $n$-th stage $\ZA_{k_0-n}$ gets countable number of balls so the family $\bar\ZA$ will also be countable at the end of whole process. 
	
	Now we shall prove that for an admissible constant $\alpha>0$ the family $\bar\ZA$ is a union of two $1/2$-sparse collections of balls. For $A\in \bar\ZA$ define
	\begin{equation}\label{d25}
		E(A)=A\setminus \bigcup_{G\in \bar\ZA:\,\r(G)\ll\r(A)} G
		=A\setminus \bigcup_{G\in \bar\ZA:\,G\cap A\neq\varnothing,\,\r(G)\ll\r(A)}G.
	\end{equation}
	Observe that 
	\begin{equation}\label{y13}
		E(A)\cap E(B)=\varnothing,\text { if } \r(A)\not \asymp \r(B) \text { or } \r(A)= \r(B).
	\end{equation}
	Indeed, take arbitrary $A,B\in \bar\ZA$. If $\r(A)=\r(B)$, then the balls $A,B$ as a result of the application of Procedure 2 (\lem {L1-1}) are pairwise disjoint. Therefore from \e {d25} it follows that $E(A)\cap E(B)=\varnothing$. If $\r(A)\gg\r(B)$, then $E(A)\cap B=\varnothing$
	immediately follows from definition \e {d25} and so we will again have $E(A)\cap E(B)=\varnothing$. 
	To prove 
	\begin{equation}\label{aa28}
		\mu(E(A))\ge \mu(A)/2
	\end{equation}
	take an arbitrary $A\in \bar\ZA$ and denote
	\begin{equation*}
		\ZP=\ZP_A=\{P\in \bar\ZA:\, \r(P)\asymp\r(A),\,P^{**}\cap A\neq\varnothing\}.
	\end{equation*} 
	We have 
	\begin{equation}\label{h55}
		R^{-2}\cdot \mu(A)\le \mu(P)\le R^2\cdot \mu(A),\quad P\in \ZP,
	\end{equation}
as well as
	\begin{equation*}
		\ZP\subset \ZA_{l-1}\cup\ZA_l\cup\ZA_{l+1},
	\end{equation*}
	where $l=\r(A)$. Hence $\ZP$ consists of three families of pairwise disjoint balls (see Procedure 2). Thus, applying  \lem {L1-2} (see also \rem{R2}), we get 
	\begin{equation}\label{a84}
		\#\ZP\le\ZK^4.
	\end{equation}
	Suppose that $G\in \bar\ZA$ satisfies
	\begin{equation*}
		\r(G)\ll\r(A),\quad G\cap A\neq\varnothing,  
	\end{equation*}
	and so $G^{**}\cap A\neq\varnothing$. Since $G$ has not been removed via Procedure 1, we have $\r(\pr^k(G))\asymp\r(A)$ for some integer $k\ge 1$. Denote $P=\pr^k(G)\in \bar \ZA$. We have $\r(P)\asymp\r(A)$ as well as by \e {y11} $P^{**}\supset G^{**}$ and so $P^{**}\cap A\neq\varnothing$. This implies that $P\in\ZP$ and $G\in \Gen(P)$.  Hence, from \e{y72}, \e{d25}, \e {h55} and  \e {a84} it follows that 
	\begin{align*}
		\mu(A\setminus E(A))&\le \mu\left(\bigcup_{G\in \bar\ZA:\,G\cap A\neq\varnothing,\,\r(G)\ll\r(A)}G\right)\\
		&\le \mu\left(\bigcup_{P\in \ZP}\bigcup_{G\in \Gen(P)}G\right)\le \sum_{P\in \ZP}\sum_{k=1}^\infty \mu\left(\bigcup_{G:\,\pr^k(G)=P}G\right)\\
		&\le \sum_{P\in \ZP}\sum_{k=1}^\infty (5\ZK^4)^{-k}\mu(P)< \mu(A)/2.  
	\end{align*}
	Thus we get \e {aa28}. From \e {y13} and \e {aa28} one can easily conclude that two families  
	\begin{align}
	&\bar\ZA_1=\{A\in\bar\ZA:\, \r(A)\text { is odd}\},\\
		& \bar\ZA_2=\bar\ZA\setminus \bar\ZA_1=	\bar\ZA_1=\{A\in\bar\ZA:\, \r(A)\text { is even}\},
	\end{align}
	are $1/2$-sparse.  Let us see that for any $A\in \bar\ZA$ the set
	\begin{equation}\label{d31}
		Q_A=\left([A]\cap F_{[[A]]}\right)\setminus\bigcup_{G\in \bar\ZA:\, \r(G)<\r(A)}G^{*}
	\end{equation}
is a null set. According to \e {s14} it  is enough to prove that 
	\begin{equation}\label{d28}
		D=\bigcup_{G\in \bar\ZA:\, \r(G)<\r(A)}G^{*}
		\supset \bigcup_{G\in\ZA:\, G\in\Ch(A)}G.
	\end{equation}
	Take $A\in\ZA$ and arbitrary $G\in\Ch(A)$. We have $\r(G)<\r(A)$. In the case $G\in \bar\ZA$, that is $G$ has not  been removed from $\ZA$ during the Procedures 1 and 2, $G$ is an element of the left union of \e {d28} and so $G\subset D$. If $G\not\in \bar\ZA$, then $G$ has been removed during the removal process. If $G$ was removed by an application of Procedure 1, then there exists a ball $B\in \bar\ZA$ such that $G^{**}\cap B\neq \varnothing$ and $\r(G)\ll\r(B)\ll\r(\pr(G))=\r(A)$ (see two remarks after Procedure 1). On the other hand for a double hull-ball of $[[G]]$ we have
	\begin{equation*}
		\mu([[G]])\le \ZK^2\mu(G)\le \ZK^2\cdot \frac{\mu(B)}{R}= \mu(B).
	\end{equation*}
	Thus we get $G\subset G^{**}\subset [[G]]\subset B^{*}$, which means $G\subset D$. If $G$ was removed by an application of Procedure 2, then according to \e {a70} we have $G\subset \cup_kG_k^{*}$ for a family of balls $G_k$ satisfying $\r(G_k)=\r(G)<\r(A)$ and so $G_k^{*}\subset D$. This again implies $G\subset D$ and so we get \e {d31}. Now observe that 
	\begin{equation*}
		E=\bigcap_{k\le k_0} \bigcup_{G\in\bar\ZA:\, \r(G)\le k}G^{*}
	\end{equation*}
	is a null-set, since $\bar\ZA$ consists of countable number of balls with bounded sum of their measures (see \e{y72}). Since the sets \e{d31} are null sets, so is $F=\cup_{A\in \bar\ZA}Q_A$. Choose an arbitrary $x\in [A_0]\setminus (E\cup F)$. From $x\notin E$ it follows that
	\begin{equation}
		x\in \bigcup_{G\in\bar\ZA:\, \r(G)\le k+1}G^{*}\setminus \bigcup_{G\in\bar\ZA:\, \r(G)\le k}G^{*}
	\end{equation}
for an integer $k$. Thus there exists a ball $A\in \bar\ZA$ (with $\r(A)=k+1$) such that
	\begin{equation}\label{y79}
		x\in A^{*}\setminus \bigcup_{G\in \bar\ZA:\,\r(G)<\r(A)}G^{*}\subset  [A]\setminus \bigcup_{G\in \bar\ZA:\,\r(G)<\r(A)}G^{*}.
	\end{equation}  
Since $x\notin F$, we can write
\begin{equation}\label{y80}
	x\notin Q_A=\left([A]\cap F_{[[A]]}\right)\setminus\bigcup_{G\in \bar\ZA:\, \r(G)<\r(A)}G^{*}.
\end{equation}
From \e{y79} and \e{y80} we conclude $x\in [A]\setminus F_{[[A]]}$. Hence we obtain
\begin{equation}\label{y74}
	A_0\subset [A_0]\subset_{a.e.}\bigcup_{A\in \bar\ZA}[A]\setminus F_{[[A]]}\subset \bigcup_{A\in \bar\ZA}[[A]]\setminus F_{[[A]]}.
\end{equation}
One can now check that $\ZS=\{[[A]]:\, A\in \bar\ZA\}$ is the desired sparse tree-collection, satisfying conditions a1)-a4), where the children relationship is the same as we had in $\bar\ZA$. Indeed, \e{y74} $\Rightarrow$ a2), \e{s15} $\Rightarrow$ a4). An intermediate inequality in \e{y72} $\Rightarrow$ a3).
From  \e{y72} and T2) condition it easily follows that $\mu([[G]])\le \mu([A])$ whenever $G\in \ch(A)$. Thus, using also \e{s9} we conclude $[[G]]\subset [[A]]$, that implies a1). Hence the proof is complete.

.

\end{proof}

\begin{proposition}\label{L13}
	Let $\ZA$ be a countable tree-collection of measurable sets in $X$ such that 
	\begin{equation}\label{y44}
		G\in \ch(A)\Rightarrow G\subset A,
	\end{equation} 
	and let $\{E_A\subset A:\, A\in \ZA\}$ be another family of measurable sets. Then there exists a martingale family of sets $\bar \ZA=\{\bar A\subset A:\, A\in \ZA\}\Subset\ZA$ such that the sets $\bar A\cap E_A$, $A\in \ZA$, are pairwise disjoint and 
		\begin{align}
		&G\in \ch(A)\Rightarrow\bar G\subset \bar A,\label{y67}\\
		&G\notin \ch(A)\Rightarrow\bar G\cap  \bar A=\varnothing,\label{y68}\\
		&\cup_{A\in \ZA}E_A=\cup_{A\in \ZA}\bar A\cap E_A.\label{y64}
	\end{align}
\end{proposition}
\begin{proof}
	Without loss of generality we can suppose that
	\begin{equation}\label{y70}
		\cup_{A\in \ZA}A=\cup_{A\in \ZA}E_A,
	\end{equation}
since otherwise one can replace every $G\in \ZA$ by $G\cap \left(\cup_{A\in \ZA}E_A\right)$.
The family $\bar \ZA$ will be constructed  inductively as follows. Since $\ZA$ is countable we can write $\ZA=\{A_k,\, k=1,2,\ldots\}$, using an arbitrary numeration as well as keeping the same children relation between the elements. We will inductively change the terms of $\ZA$ as follows. At the $n$-th stage of induction every element $A\in \ZA$ will be replaced by $A\setminus A_n\cap E_{A_n}$ provided $A_n\notin \{A\}\cup \Gen(A)$. Finishing the induction procedure we will finally get the desired sequence $\bar \ZA=\{\bar A_k,\, k=1,2,\ldots\}$ as a limiting version of the sequence $\ZA_n$. Let us check the claims of the lemma. 1) To prove that $\bar A_k\cap E_{A_k}$ are pairwise disjoint let $\bar A_n$ and $\bar A_j$ be two different elements of $\bar \ZA$. Without loss of generality we can suppose that $A_n\notin \Gen(A_j)$. So one can check that  after the $n$-th stage of induction $A_n\cap E_{A_n}$ and $A_j\cap E_{A_j}$ became disjoint, staying disjoint during the next stages of induction. 2) For \e{y64} one just need to check that the union $\cup_{A\in \ZA}A\cap E_A$ is not changed after any stage of induction. 3) To prove relation \e{y67} observe that for two sets $G\in \ch(A)$ the relation $G\subset A$
that we have at the beginning, will be satisfied after any stage of induction.  Indeed, if we are at the $n$-th stage and $A_n\in \{A\}\cup \Gen(A)$, then $A$ is not changed by the $n$-th stage action.  If $A_n\notin \{A\}\cup \Gen(A)$, then both $G$ and $A$ are reduced similarly. So in both cases we will still have $G\subset A$. 4) To show \e{y68} we need to see that if $A$ and $B$ are not in a parental relationship, then $\bar A\cap \bar B=\varnothing $. By \e{y64} and \e{y70} it is enough to show that
	\begin{equation}\label{y66}
	\bar A\cap \bar B\cap (\bar A_n\cap E_{A_n})=\varnothing
	\end{equation}
	for any $n=1,2,\ldots$. It is clear that each $A_n$ is not in the generation of either $A$ or $B$. Thus the set $\bar A_n\cap E_{A_n}$ has been removed at least from one of the sets $A$ or $B$ at the $n$-th stage of the induction, which yields \e{y66}. With this we complete the proof of lemma.
\end{proof}

\section{Truncated $\BO$ operators}\label{S4}
\subsection {The fractional maximal function}
 Let $(X,\mu)$ be a measure space with a ball-basis $\ZB$. Given $r\ge 0$ and $ \varrho\ge \rho>0$ define the fractional maximal function
\begin{equation}\label{1-1}
	\MM f(x)=\sup_{B\in \ZB:\, x\in B}\langle f\rangle_B=\sup_{B\in \ZB:\, x\in B}\frac{1}{(\mu(B))^\rho} \left(\int_B\|f\|_\ZU^r\right)^\varrho
\end{equation}
associated with a ball-basis $\ZB$.  
\begin{theorem}\label{T1-1}
	The maximal operator \e {1-1} satisfies the bound
	\begin{equation}\label{h38}
		\mu\left\{x\in X:\, \MM f(x)>\lambda\right\}\le \frac{\ZK}{\lambda^{1/\rho}}\left(\int_{X }\| f\|_\ZU^r\right)^{ \varrho/\rho},\quad \lambda>0,
	\end{equation} 
\end{theorem}
\begin{proof}[Proof of \trm{T1-1}] 
Consider the set $E=\{ x\in X:\, \MM f(x)>\lambda\}$, which can be non-measurable. For any $x\in  E $ there exists a ball $B(x)\subset X$ such that
	\begin{equation*}
		x\in B(x),\quad \frac{1}{\mu(B(x))^{\rho}}\left(\int_{B(x)}\|f\|_\ZU^r\right)^{ \varrho}>\lambda,
	\end{equation*}
which implies
\begin{equation}
	\mu(B(x))<\frac{1}{\lambda^{1/\rho}}\left(\int_{B(x)}\|f\|_\ZU^r\right)^{ \varrho/\rho}.
\end{equation}
	We have $E=\cup_{x\in E} B(x)$. Given $G\in\ZB$ consider the collection of balls $\{B(x):\, x\in E\cap G\}$. Applying \lem{L1-1}, we find a sequence of pairwise disjoint balls $\{B_k\}$ taken from this collection such that $E \cap G\subset \cup_kB_k^*=Q(G)$.
	Applying the inequality $\sum_kx_k^d\le (\sum_kx_k)^d$ for positive numbers $x_k$ and for $d=\varrho/\rho\ge 1$, we get
	\begin{align}
		\mu(E\cap G)&\le \mu(Q(G))	\le \sum_k \mu(B_k^*)\le \ZK \sum_k \mu(B_k)\\
		&\le \frac{\ZK}{\lambda^{1/\rho}}\sum_k \left(\int_{B_k}\|f\|_\ZU^r \right)^{ \varrho/\rho}\le \frac{\ZK}{\lambda^{1/\rho}}\left(\int_{X }\|f\|_\ZU^r\right)^{ \varrho/\rho}.
	\end{align}
	According to \lem {L5} there is a sequence of balls $G_1\subset G_2\subset \ldots $ such that $X=\cup_{k}G_k$,
	so we conclude
	\begin{equation*}
		\mu(E)=\lim_{n\to\infty }\mu(E\cap G_n)\le \frac{\ZK}{\lambda^{1/\rho}}\left(\int_{X }\|f\|_\ZU^r\right)^{ \varrho/\rho}
	\end{equation*}
	and so \e {h38}. 
\end{proof}

\subsection{Inequalities related to T2) condition}
We say a measurable set $E\subset X$ is a $\gamma$-ball if there exist balls $B_1,B_2$,  with $B_1\subset E\subset B_2$, such that $\mu(B_2)\le \gamma \mu(B_1)$. For example for any ball $B$ the sets $B^*$ and $B^{**}$ are $\ZK$ and respectively $\ZK^2$ balls.
\begin{lemma}\label{L14}
	If a sublinear operator $T$ satisfies T1), then for any $f\in L^r(X,\ZU)$  and a $\gamma$-ball $B$ we have 
		\begin{equation}\label{y60}
			\OSC_B\big(T(f\cdot \ZI_{X\setminus B^*})\big) \lesssim\gamma^\rho \ZL_1(T)\langle f \rangle^*_{B},
		\end{equation}
\end{lemma}
\begin{proof}
	Let $B$ be a $\gamma$-ball so there are balls $B_1\subset B\subset B_2$ such that $\mu(B_2)\le \gamma \mu(B_1)$.
	Thus for $f\in L^r(X,\ZU)$ we can write
	\begin{align}
		\OSC_B\big(T(f\cdot \ZI_{X\setminus B^*})\big)&=\OSC_B\big(T(f\cdot \ZI_{X\setminus B^*}\cdot \ZI_{X\setminus B_1^*})\big)\\
		&\le \ZL_1\cdot \langle f\cdot \ZI_{X\setminus B^*}\rangle^*_{B_1}\le2 \ZL_1\cdot \langle f\cdot \ZI_{X\setminus B^*}\rangle_{\bar B_1}
	\end{align}
	for some ball $\bar B_1\supset B_1$. If $\mu(\bar B_1)\le \mu(B_2)$, then we have $\bar B_1\subset B_2^*$ and therefore
	\begin{equation*}
		\langle f\cdot \ZI_{X\setminus B^*}\rangle_{\bar B_1}\le  \left(\frac{\mu(B_2^*)}{\mu( \bar B_1)}\right)^\rho\langle f\cdot \ZI_{X\setminus B^*}\rangle_{B_2^*}\le (\gamma\ZK)^\rho\langle f \rangle^*_{B}.
	\end{equation*}
	Otherwise, if $\mu(\bar B_1)\ge \mu(B_2)$, then we will have 
	$B\subset B_2\subset \bar B_1^{*}$ and then
	\begin{equation}
		\langle f\cdot \ZI_{X\setminus B^*}\rangle_{\bar B_1}\le \ZK^\rho\langle f\cdot \ZI_{X\setminus B^*}\rangle_{\bar B_1^*}\le \ZK^\rho\langle f\rangle_{B}^*.
	\end{equation}
	Thus we obtain \e{y60}.
\end{proof}
\begin{definition}
	Let $T$ be a sublinear operator. Given measurable sets $A$ and $B\supset A$ we denote
	\begin{equation}\label{y17}
		\Delta(A,B)=\Delta_T(A,B)=\sup_{x\in A,\, f\neq 0\in L^r(X,\ZU)}\frac{\|T(f\cdot \ZI_{B^{*}\setminus A^{*}})(x)\|_\ZV}{\langle f\rangle_{B^{*}}}.
	\end{equation}
\end{definition}
Notice that T2)-condition for a sublinear operator $T$ means that for any $A\in \ZB$ there exists a ball $B\supsetneq A$ such that $\Delta(A,B)\le \ZL_2$. The following lemma clearly follows from definition \e{y17}.
\begin{lemma}\label{L19}
	If $T$ is an arbitrary sublinear operator, then for any balls $A$, $B$ and $C$ satisfying $A\subset B\subset C$ we have
	\begin{equation}\label{a63}
		\Delta(A,B)\le\Delta(A,C).
	\end{equation}
\end{lemma}
\begin{proof}
	Let $f\in L^r(X,\ZU)$, $A\subset B\subset C$ and $x\in B$. We have
	\begin{align}
		\|T(f\cdot \ZI_{B^*\setminus A^*})(x)\|_\ZV&=	\|T(f\cdot \ZI_{B^*\setminus A^*}\cdot \ZI_{C^*\setminus A^*})(x)\|_\ZV\\
		&\le \Delta(A,C)\langle f\cdot \ZI_{B^*\setminus A^*}\rangle_{C^*}\le \Delta(A,C)\langle f\rangle_{B^*}
	\end{align}
that completes the proof.
\end{proof}
	\begin{lemma}\label{L17}
	If a sublinear operator $T$ satisfies T0) and T1)-conditions, then for any $\gamma$-balls $A,B$ and $C$, satisfying $A\subset B\subset C$, we have
	\begin{equation}\label{h51}
		\Delta(A,C)\lesssim\left(\frac{\mu(C)}{\mu(B)}\right)^{\rho}\left(\ZL_0(T)+\gamma^\rho\ZL_1(T)+\Delta(A,B)\right).
	\end{equation}
\end{lemma}
	\begin{proof}
	First consider the case when $A=B$ and so $\Delta(A,B)=0$. Applying \e{y57} with 
	\begin{equation}
		\lambda= \ZL_0\cdot \left(\frac{2\mu(C^*)}{\mu(B)}\right)^\rho\langle f\rangle_{C^*}= \ZL_0\cdot \left(\frac{2}{\mu(B)}\right)^\rho \left(\int_{C^*}\|f|\|_\ZU^r\right)^{ \varrho}
	\end{equation}
	 we get
	\begin{align*}
		&\mu\left\{x\in B:\, \|T(f\cdot\ZI_{C^{*}\setminus B^{*}})(x)\|_\ZV> \ZL_0\cdot \left(\frac{2\mu(C^*)}{\mu(B)}\right)^\rho\langle f\rangle_{C^*}\right\}\le \frac{\mu(B)}{2}
	\end{align*}	
	and so we find a point $x_0\in B$ such that
	\begin{align}
		\|T(f\cdot\ZI_{C^{*}\setminus B^{*}})\}(x_0)\|_\ZV&\le  \ZL_0\cdot \left(\frac{2\mu(C^*)}{\mu(B)}\right)^\rho\langle f\rangle_{C^*}\\
		&\lesssim  \ZL_0\cdot \left(\frac{\mu(C)}{\mu(B)}\right)^\rho\langle f\rangle_{C^*}.\label{f1}
	\end{align}
	According to T1)-condition and \lem{L14}, for any $x\in B$ we also have 
	\begin{equation}\label{f2}
		\|T(f\cdot\ZI_{C^{*}\setminus B^{*}})(x)-T(f\cdot\ZI_{C^{*}\setminus B^{*}})(x_0)\|_\ZV\le \gamma^\rho\ZL_1 \langle f\cdot\ZI_{C^{*}\setminus B^{*}} \rangle^*_{B}.
	\end{equation}
	By the definition of $\langle f \rangle^*_{B}$ there is a ball $G\supset B$ such that
	\begin{equation}\label{y20}
		\langle f\cdot\ZI_{C^{*}\setminus B^{*}} \rangle^*_{B}<2\langle f\cdot\ZI_{C^{*}\setminus B^{*}}\rangle_{G}.
	\end{equation}
	If $\mu(G)\le \mu(C)$, then we have $G\subset C^{*}$ and therefore 
	\begin{align}
		\langle f\cdot\ZI_{C^{*}\setminus B^{*}}\rangle_{G}\le \left(\frac{\mu(C^{*})}{\mu(G)}\right)^{\rho}\langle f\cdot\ZI_{C^{*}\setminus B^{*}}\rangle_{C^{*}}\lesssim\left(\frac{\mu(C)}{\mu(B)}\right)^{\rho}\cdot \langle f\rangle_{C^{*}}.\label{h34}
	\end{align}
	If $\mu(G)> \mu(C)$, then $C^*\subset G^{**}$. Hence we get
	\begin{align}
		\langle f\cdot\ZI_{C^{*}\setminus B^{*}}\rangle_{G}&\le \left(\frac{\mu(G^{**})}{\mu(G)}\right)^{\rho}\langle f\cdot\ZI_{C^{*}\setminus B^{*}}\rangle_{G^{**}}\\
		&\lesssim\langle f\cdot\ZI_{C^{*}\setminus B^{*}}\rangle_{C^{*}}\le \langle f\rangle_{C^{*}}.\label{h35}
	\end{align}
	The combination of \e{y20}, \e {h34} and \e {h35} imply the inequality 
	\begin{equation*}
		\langle f\cdot\ZI_{C^{*}\setminus B^{*}} \rangle^*_{B}\lesssim \left(\frac{\mu(C)}{\mu(B)}\right)^{\rho}\cdot \langle f\rangle_{C^{*}},
	\end{equation*}
	which together with \e {f1} and \e {f2} gives
	\begin{equation*}
		\|T(f\cdot\ZI_{C^{*}\setminus B^{*}})(x)\|_\ZV\lesssim\left(\frac{\mu(C)}{\mu(A)}\right)^{\rho} (\ZL_0+\gamma^\rho\ZL_1)\langle f\rangle_{C^{*}},\quad x\in A.
	\end{equation*}
The latter implies
\begin{equation*}
	\Delta(B,C)\lesssim \left(\frac{\mu(C)}{\mu(B)}\right)^{\rho}(\ZL_0+\gamma^\rho\ZL_1)
\end{equation*}
that is \e{h51} if $A=B$. Now let $A$ be arbitrary.  Thus, for $f\in L^r(X,\ZU)$ and $x\in A$ we have
	\begin{align*}
		\|T(f\cdot &\ZI_{C^{*}\setminus A^{*}})(x)\|_\ZV\\
		&\le \|T(f\cdot \ZI_{C^{*}\setminus B^{*}})(x)\|_\ZV+\|T(f\cdot \ZI_{B^{*}\setminus A^{*}})(x)\|_\ZV\\
		&\lesssim\Delta(B,C)\langle f\rangle_{C^{*}}+\Delta(A,B) \langle f\rangle_{B^{*}}\\
		&\lesssim \left(\frac{\mu(C)}{\mu(B)}\right)^{\rho}(\ZL_0+ \gamma^\rho\ZL_1)\langle f\rangle_{C^{*}}\\
		&\qquad +\left(\frac{\mu(C^{*})}{\mu(B^{*})}\right)^{\rho}\Delta(A,B) \langle f\rangle_{C^{*}}\\
		&\lesssim \left(\frac{\mu(C)}{\mu(B)}\right)^{\rho }(\ZL_0+\gamma^\rho\ZL_1+ \Delta(A,B))\langle f\rangle_{C^{*}},
	\end{align*}
	which is the full version of \e {h51}.
\end{proof} 
\begin{remark}
	We will often use the following particular case of inequality \e{h51}, that is 
		\begin{equation}\label{y22}
			\Delta(A,B)\lesssim\left(\frac{\mu(B)}{\mu(A)}\right)^{\rho}\left(\ZL_0+ \gamma^\rho\ZL_1\right).
		\end{equation}
\end{remark}
\begin{lemma}\label{L15}
	If $T$ is a $\BO$ operator, then for any $\gamma$-ball $A\in \ZB$ ($A^*\neq X$) there exists a ball $B\supsetneq A$ such that 
		\begin{equation}\label{y62}
			\Delta(A,B)\lesssim \gamma^{\rho}\ZL_0+\gamma^{2\rho}\ZL_1+\ZL_2,
		\end{equation}
\end{lemma}
whenever $f\in L^r(X,\ZU)$.
\begin{proof}
Let $A$ be a $\gamma$-ball that is $A_1\subset A\subset A_2$, $\mu(A_2)\le \gamma \mu(A_1)$ for some balls $A_1,A_2$.  If $A=A_2$, then by T2) we will have \e{y62} with a constant $\ZL_2$. If $A\neq A_2$, then, applying \e{y22} we obtain
	\begin{equation}
\Delta(A,A_2)\lesssim\left(\frac{\mu(A_2)}{\mu(A)}\right)^{\rho}\left(\ZL_0+ \gamma^\rho\ZL_1\right)\le \gamma^{\rho}\ZL_0+\gamma^{2\rho}\ZL_1.
	\end{equation}
Thus we get \e{y62} with $B=A_2$.
\end{proof}
\begin{remark}\label{R1}
	Taking into account Lemmas \ref{L14} and \ref{L15}, in the definition of $\BO$ operators we can equivalently consider $\lambda$-balls instead of real balls.
\end{remark}
For any a ball $A\in \ZB$, $A\neq X$, denote
\begin{equation*}
	\phi(A)=\inf_{B\in \ZB:\, B\supsetneq A}\mu(B).
\end{equation*}
From \lem {L5} it easily follows that the set of balls $B$ satisfying  $B\supsetneq A$ is nonempty and so the number $\phi(A)$ is precisely determined as $A\neq X$. 
\begin{lemma}\label{L16}
	Let $T$ is a $\BO$ operator and $A$ be a ball. Then for any ball $B\supsetneq A$ such that $\mu(B)\le 2\phi(A)$ we have
	\begin{equation}
		\Delta(A,B)\lesssim \ZL_0(T)+\ZL_1(T)+\ZL_2(T).
	\end{equation}
\end{lemma}
\begin{proof}
	Let balls $B\supsetneq A$ satisfy the conditions of lemma.  Using T2)-condition, we find a ball $C\supsetneq A$ such that
	\begin{equation}\label{t6}
	\Delta(A,C)\le\ZL_2(T).
	\end{equation}
If $\mu(C)>\mu(B)$, then we have $B\subset C^{*}$. Thus, applying Lemmas \ref {L19} and \ref{L17}, we obtain 
\begin{align}
	\Delta(A,B)\le \Delta(A,C^*)&\lesssim \ZL_0(T)+\ZL_1(T)+\Delta(A,C)\\
	&\le \ZL_0(T)+\ZL_1(T)+\ZL_2(T).\label{t8}
\end{align}
	In the case $\mu(C)\le \mu(B)$ we have $C\subset B^{*}$, and since $C\supsetneq A$, we obtain 
	\begin{equation*}
		\mu(B)\gtrsim \mu(B^{*})\ge \mu(C)\ge \phi(A) \ge \frac{\mu(B)}{2}.
	\end{equation*} 
	Thus, once again applying Lemmas \ref {L19}, \ref{L17} and \e {t6}, we conclude 
	\begin{align}
			\Delta(A,B)&\le 	\Delta(A,B^*)\lesssim \ZL_0(T)+\ZL_1(T)+	\Delta(A,C)\\
			&\le \ZL_0(T)+\ZL_1(T)+\ZL_2(T).
	\end{align}
This completes the proof of lemma.
\end{proof}

\begin{lemma}\label{L18}
	Let $T$ be a $\BO$ operator. For any balls $A$ and $B$ such that $A\cap B\neq \varnothing$ and $\mu(A)<\mu(B)$ there exists a ball $\tilde A$ such that $A\subsetneq \tilde A$, $\tilde A^*\subset B^*$ and 
	\begin{equation}\label{y23}
		\Delta(A,\tilde A)\lesssim \ZL_0+\ZL_1+\ZL_2
	\end{equation}
	One can also write $A^*$ and $\tilde A^*$ in \e{y23} instead of $A$ and $\tilde A$ respectively.
\end{lemma}
\begin{proof}
	Let a pair $(G,A)$ of balls $G\supsetneq A$ satisfy T2) condition, that is $\Delta(A,G)\le \ZL_2$. If for a hull-ball $[G]$ we have $\mu([G])\le \mu(B)$, then $G^*\subset [G]\subset B^*$ and one can choose $\tilde A=G$ in \e{y23}. In the case of $\mu([G])>\mu(B)$ we have $B\subset [G]^{*}$. By \lem{L17} the pair $(A, [G]^{*})$ also satisfies T2)-condition, and according to \lem{L19}, so we have for $(A,B)$. Hence we can choose $\tilde A=B$. One can check that for both cases we will have \e{y23}. The second part of the lemma may be proved similarly.
\end{proof}
\subsection{Truncated $\BO$ operators}
	For a $\BO$ operator \e{y58} we consider the truncated operator
	\begin{equation*}
		T^*f(x)=\sup_{B\in \ZB:\, x\in  B}\|T(f\cdot \ZI_{X\setminus B^*})(x)\|_\ZV.
	\end{equation*}	
\begin{theorem}\label{T5-1}
	If a sublinear operator $T$ satisfies T0) and T1)-conditions, then 
	\begin{equation}\label{y14}
		\mu\{x\in X:\, T^*f(x)>\lambda,\, \MM f(x)< \delta \lambda\}\lesssim \mu\{x\in X:\, \|Tf(x)\|_\ZV>\lambda/2\}
	\end{equation}
	with a constant $\delta\sim (\ZL_1(T)+ \ZL_0(T))^{-1}$.
\end{theorem}
\begin{proof}
	Denote by $E=E_\lambda$ the set on the left of \e{y14}, which can also be non-measurable. By the definition of $T^*$ for any $x\in E$ there is a ball $B(x)\in\ZB$ such that 
	\begin{equation}\label{h27}
		x\in B(x),\quad \|T(f\cdot \ZI_{X\setminus B^{*}(x)})(x)\|_\ZV>\lambda.
	\end{equation}
	One can check that $E\subset \cup_{x\in E}B(x)$ and
	\begin{equation}\label{y15}
		\langle f\rangle_{B(x)}^*\le \MM f(x)<\delta \lambda,\quad x\in E.
	\end{equation}
	Given ball $G$ apply \lem {L1-1}, we find a sequence $x_k\in E$ such that the balls $\{B_k=B(x_k)\}$ are pairwise disjoint and 
	\begin{equation}\label{h30}
		E\cap G \subset \cup_kB^{*}_k=Q(G).
	\end{equation}
	Since $T$ satisfies T1)-condition, by \e{y15} we obtain 
	\begin{equation*}
		\|T(f\cdot \ZI_{X\setminus B^{*}_k})(x_k)-T(f\cdot \ZI_{X\setminus B^{*}_k})(x)\|_\ZV\le \ZL_1\cdot \langle f \rangle^*_{B_k}\le \ZL_1\delta\lambda,\quad x\in B_k. 
	\end{equation*}
	Thus, we conclude from \e {h27} that
	\begin{align}
		\|T(f\cdot \ZI_{X\setminus B^{*}_k})(x)\|_\ZV&\ge \|T(f\cdot \ZI_{X\setminus B^{*}_k})(x_k)\|_\ZV\label{h29}\\
		&\qquad -\|T(f\cdot \ZI_{X\setminus B^{*}_k})(x_k)-T(f\cdot \ZI_{X\setminus B^{*}_k})(x)\|_\ZV \\
		&\ge \lambda(1-\ZL_1\delta ),\quad x\in B_k.
	\end{align}
	For a $\beta>0$ we define measurable sets
	\begin{equation}\label{h28}
		\tilde B_k= \{x\in B_k:\, \|T(f\cdot \ZI_{B_k^{*}})(x)\|_\ZV\le \beta\cdot \langle f \rangle^*_{B_k,r}\}.
	\end{equation} 
	Using T0)-condition, the measure of the complement of $\tilde B_k$ is estimated by
	\begin{align*}
		\mu(\tilde B_k^c)&\le \frac{ ( \ZL_0)^{1/\rho}}{(\beta \cdot \langle f \rangle^*_{B_k,r})^{1/\rho}}\cdot \left(\int_{B_k^{*}}|f|^r\right)^{ \varrho/\rho}\le \left(\frac{ \ZL_0}{\beta}\right)^{1/\rho}\mu(B_k^{*}) \\
		&\lesssim \left(\frac{ \ZL_0}{\beta}\right)^{1/\rho}\mu(B_k).
	\end{align*}
	Using also \e{y15}, for an appropriate constant $\beta\sim  \ZL_0$ we have
	\begin{align}
	&	\|T(f\cdot \ZI_{B_k^{*}})(x)\|_\ZV\le \beta \delta\lambda,\quad x\in \tilde B_k,\label{y21}\\
	&	\mu(\tilde B_k)\ge\mu(B_k)-\mu(\tilde B_k^c) \ge \mu(B_k)/2.\label{z43}
	\end{align}
	If $x\in \tilde B_k$, then, using sublinearity of $T$ together with relations \e {h29} and \e {y21}, we obtain
	\begin{align*}
		\|Tf(x)\|_\ZV&\gtrsim \|T(f\cdot \ZI_{X\setminus B_k^{*}})(x)\|_\ZV-\|T(f\cdot \ZI_{B_k^{*}})(x)\|_\ZV\\
		&\ge \lambda(1-\ZL_1\delta-\beta\delta)\ge \lambda/2,   	
	\end{align*} 
	where the last inequality can be satisfied for $\delta= 1/2(\ZL_1+\beta)\sim (\ZL_1+ \ZL_0)^{-1}$.
	Hence we conclude $\cup_k\tilde B_k\subset \{\|Tf(x)\|_\ZV>\lambda/2\}$
	and taking into account also \e {h30} and \e {z43}, we get
	\begin{align*}
		\mu(E\cap G)\le \mu(Q(G))&\le \sum_k\mu(B_k^{*})\le \ZK\cdot \sum_k\mu(B_k)\\
		&\le 2\ZK\cdot  \sum_k\mu(\tilde B_k)\le 2\ZK\mu \{\|Tf(x)\|_\ZV>\lambda/2\}.
	\end{align*}
	Since $G$ is arbitrary, applying \lem{L5} we complete the proof. 
\end{proof}
\begin{theorem}\label{T6}
	If $T$ is a $\BO$ operator, then $T^*$ is also a $\BO$ operator. Moreover, we have
	\begin{align}
		&\ZL_0(T^*)\lesssim \ZL_0(T)+ \ZL_1(T),\label{y59}\\
		&\ZL_1(T^*)\lesssim \ZL_0(T)+ \ZL_1(T),\label{y78}\\
		&\ZL_2(T^*)=\ZL_2(T),		
	\end{align}
\end{theorem}
\begin{proof}
Weak type inequality \e{y57} or condition \e{y59} easily follows from Theorems \ref{T1-1} and \ref{T5-1}.	One can also check that for any balls $A,B$ satisfying $A\subset B$ we have 
	\begin{equation*}
		\Delta_{T^*}(A,B)=\Delta_{T}(A,B).
	\end{equation*}
	Thus, if balls $A$ and $B$ satisfy T2)-condition for the operator $T$, then the same conditions hold also for $T^*$ with $\ZL_2(T^*)=\ZL_2(T)$. 
	To prove \e{y78}, let $B\in \ZB$ and $f\in L^r(X,\ZU)$ satisfy 
	\begin{equation}\label{a67}
		\supp f\in X\setminus B^{*}.
	\end{equation}
	Choose arbitrary points $x,x'\in B$ and let estimate $|T^*f(x)- T^*f(x')|$. If $T^*f(x)= T^*f(x')$, then the estimation is trivial. So we can suppose that $T^*f(x)> T^*f(x')$. Using the definition of $T^*f(x)$, we find a ball $A\ni x$ such that
	\begin{equation}\label{h37}
		\frac{T^*f(x)+T^*f(x')}{2}< \|T(f\cdot \ZI_{X\setminus A^{*}})(x)\|_\ZV.
	\end{equation} 
First suppose that $\mu(A^*)>\mu(B)$. So we can say that $x'\in B\subset A^{**}\subset A'$, where  $A'$ is a ball satisfying $\mu(A')\le \ZK^2\mu(A)$. By definition of $T^*$ we can write
\begin{equation}
	T^*f(x')\ge \|T(f\cdot \ZI_{X\setminus A'^{*}})(x')\|_\ZV,
\end{equation}
which together with \e {h37} yields 
\begin{align}
	|T^*f(x)-T^*f(x')|&=T^*f(x)-T^*f(x')\\
	&< 2\|T(f\cdot \ZI_{X\setminus A^{*}})(x)\|_\ZV-2T^*f(x')\\
	&\le 2\bigg(\|T(f\cdot \ZI_{X\setminus A^{*}})(x)\|_\ZV-\|T(f\cdot \ZI_{X\setminus A'^{*}})(x')\|_\ZV\bigg).\label{a6}
\end{align}
Besides, since $x\in A$, from \e{y22} (or \lem {L17}) and \e {a67} it follows that
\begin{align}
	\|T(f\cdot \ZI_{A'^*\setminus A^{*}})(x)\|_\ZV&\le \Delta(A,A') \langle f\cdot \ZI_{A'^*\setminus A^{*}}\rangle_{A'^*}\\
&\lesssim \left(\frac{\mu(A')}{\mu(A)}\right)^{\rho}(\ZL_0(T)+ \ZL_1(T)) \langle f \rangle_{A'^*},\\
	&\lesssim (\ZL_0(T)+ \ZL_1(T)) \langle f \rangle^*_{B},\label{a81}
\end{align}
Thus, applying the T1)-condition for $T$, from \e {a81} and \e {a6} we conclude
\begin{align*}
	|T^*f(x)-T^*f(x')|&\le 2\bigg(\|T(f\cdot \ZI_{X\setminus A'^*})(x)\|_\ZV+\|T(f\cdot \ZI_{A'^{*}\setminus A^{*}})(x)\|_\ZV\\
	&\qquad -\|T(f\cdot \ZI_{X\setminus A'^*})(x')\|_\ZV\bigg)\\
	&\lesssim \|T(f\cdot \ZI_{X\setminus A'^*})(x)-T(f\cdot \ZI_{X\setminus A'^*})(x')\|_\ZV\\
	&\qquad +(\ZL_0(T)+ \ZL_1(T))\langle f \rangle^*_{B} \\
	&\lesssim (\ZL_0(T)+ \ZL_1(T)) \langle f \rangle^*_{B}
\end{align*}
that implies \e{y78}.  
Now suppose $\mu(A^*)\le \mu(B)$ and therefore $A^*\subset B^*$. Then by \e{a67} we get $f\cdot \ZI_{X\setminus A^{*}}=f\cdot \ZI_{X\setminus B^{*}}$. Thus, using the same argument, we can write
\begin{align}
	|T^*f(x)-T^*f(x')|&=T^*f(x)-T^*f(x')\\
	&\le 2\bigg(\|T(f\cdot \ZI_{X\setminus B^{*}})(x)\|_\ZV-\|T(f\cdot \ZI_{X\setminus B^{*}})(x')\|_\ZV\bigg)\\
	&\le  \ZL_1(T)\langle f \rangle^*_{B}
\end{align}
that again gives \e{y78}.
\end{proof}

\begin{theorem}\label{T16}
 If $\{T_\alpha\}$ is a family of $\BO$ operators with uniformly bounded characteristic constants $\ZL_j(T_\alpha)$, then the maximal operator
 \begin{equation}\label{aa4}
 	Tf(x)=\sup_\alpha \|T_\alpha f(x)\|_\ZV
 \end{equation} 
satisfies the bounds
	\begin{align}
		&\ZL_1(T)\lesssim \sup_\alpha \ZL_1(T_\alpha),\label{aa5}\\
		&\ZL_2(T)\lesssim \sup_\alpha(\ZL_0(T_\alpha)+\ZL_1(T_\alpha)+\ZL_2(T_\alpha)).\label{aa6}
	\end{align}
\end{theorem}
\begin{proof}
	Let $A\in\ZB$ be an arbitrary ball and  $\supp f\subset X\setminus A^{*}$.
	Take arbitrary points $x,x'\in A$ and suppose that $Tf(x)\ge Tf(x')$.
	According to the definition of $T$, for any $\delta> 0$ there exists an index $\alpha_0$ such that $Tf(x)\le \|T_{\alpha_0}f(x)\|_\ZV+\delta$.
	On the other hand for the same $\alpha_0$ we have $Tf(x')\ge \|T_{\alpha_0} f(x')\|_\ZV$. Thus, applying T1) property of $T_{\alpha_0}$, we obtain
	\begin{align*}
		|Tf(x)-Tf(x')|&=Tf(x)-Tf(x')\\
		&\le \|T_{\alpha_0} f(x)\|_\ZV+\delta-\|T_{\alpha_0} f(x')\|_\ZV\\
		&\le \|T_{\alpha_0} f(x)-T_{\alpha_0}f(x')\|_\ZV+\delta\\
		&\le \sup_\alpha\ZL_1(T_\alpha)\cdot \langle f \rangle^*_{A}+\delta.
	\end{align*}
	Since $\delta>0$ can be arbitrarily small, we get \e {aa5}. \e{aa6} follows from \lem{L16}, according which we can choose a ball $B$ independent of the operators $T_\alpha$ such that 
	\begin{equation}
			\Delta_{T_\alpha}(A,B)\lesssim \ZL_0(T_\alpha)+\ZL_1(T_\alpha)+\ZL_2(T_\alpha)
	\end{equation}
for any $\alpha$. This implies \e{aa6}, completing the proof of theorem.
\end{proof}
\begin{corollary}
	If $\{T_\alpha\}$ is a family of $\BO$ operators with uniformly bounded characteristic constants $\ZL_j(T_\alpha)$ and the maximal operator $T$ in \e{aa4} satisfies condition T0), then $T$ is a $\BO$ operator.
\end{corollary}
\begin{theorem}
	Let the ball-basis $\ZB$ satisfy the doubling condition. If a sublinear operator $T$ satisfies T0) and T1) conditions, then $T$ satisfies also T2) and so it is $\BO$ operator. Moreover, 
	\begin{equation}\label{y63}
		\ZL_2(T)\lesssim\ZL_0(T)+ \ZL_1(T).
	\end{equation}
\end{theorem} 
\begin{proof}
	We need to check T2)-condition. Let $A$ be an arbitrary ball and $B$ satisfy conditions \e {h73}. Then applying \e{y22}, we get the bound
	\begin{equation*}
		\Delta(A,B)\lesssim \left(\frac{\mu( B)}{\mu(A)}\right)^{\rho}(\ZL_0+\ZL_1)\le \eta^{\rho} (\ZL_0+\ZL_1),
	\end{equation*}
which implies \e{y63}.
\end{proof}
\section{Proof of \trm{T7}}\label{S5}
\begin{proof}
 Consider the operator 
	\begin{align*}
		\Gamma f(x)=\max\left\{\|Tf(x)\|_\ZV,T^{*}f(x),\ZL\cdot \MM f(x)\right\},
	\end{align*}
where $\ZL=\ZL_0(T)+\ZL_1(T)+\ZL_2(T)$. As $T$ is a $\BO$ operator, from Theorems \ref {T1-1} and \ref {T6} it follows that the operator $\Gamma$ satisfies the weak type property and $\ZL_0(\Gamma)\lesssim \ZL$. Thus for every ball $A\in \ZB$ the set
\begin{equation}\label{y75}
	F_A=\left\{x\in X:\, \Gamma (f\cdot \ZI_{A^{**}})(x)> \ZL \lambda \langle f\rangle_{A^{**}}\right\},\quad \lambda>0,
\end{equation}
has a measure 
\begin{equation}\label{y69}
	\mu(F_A)\lesssim \mu(A)/\lambda^{1/\rho}=\alpha\mu(A).
\end{equation}
We consider the sets $F_A$ for a fixed function $f\in L^r(X,\ZU)$ with $\supp f\subset B\in\ZB$. Applying \pro{P11} with an admissible constant $\lambda >0$ and $A_0=B$, we find a sparse collection of balls $\ZS$ satisfying the conditions of \pro{P11}. Then, applying \lem{L18},  for each $A\in \ZS$ and $G\in \ch(A)$ we may fix a ball $\tilde G$ such that $G\subsetneq \tilde G$, $\tilde G^*\subset A^*$ and $\Delta(G,\tilde G)\lesssim \ZL$. From \lem{L17} we get $\Delta(G,\tilde G^*)\lesssim \ZL$ and therefore for any $x\in G$ we can write
\begin{align}
	\|T(f\cdot \ZI_{\tilde G^{**}\setminus G^{**}})(x)\|_\ZV&=\|T(f\cdot \ZI_{\tilde G^{**}\setminus G^{**}}\cdot\ZI_{\tilde G^{**}\setminus G^{*}} )(x)\|_\ZV\\
	&\le \Delta(G,\tilde G^*)\langle f\cdot \ZI_{\tilde G^{**}\setminus G^{**}}\rangle_{\tilde G^{**}}\\
	&\le \ZL \langle f\rangle_{\tilde G^{**}}, \quad x\in G.\label{y76}
\end{align}
Having a1) from \pro{P11}, we can apply \pro{L13} to the tree-collection  $\ZS$ with $\{E_A=A\setminus F_A\}$. Then we get a martingale-collection $\{\bar A:\, A\in \ZS\}\Subset \ZS$ such that $A\setminus F_A$ are pairwise disjoint and \e{y67}-\e{y64} hold. Moreover, \e{y64} may be written in the form
\begin{equation}
	A_0\subset_{a.e.}\bigcup_{A\in \ZS}( A\setminus F_A)=\bigcup_{A\in \ZS}(\bar A\setminus F_A),
\end{equation}
where the first inclusion follows from condition a2) of \pro{P11}. Thus for a.e. $x\in A_0$ one can find a sequence of balls $A_k$, $k=1,2,\ldots,n$ such that $x\in \bar A_n\setminus  F_{A_n}$ and $A_{j+1}\in \ch(A_j)$, $j=0,1,\ldots,n-1$. Thus we obtain
\begin{align}
	&	x\in \bar A_n\setminus  F_{A_n}\subset \bar A_{j},\quad j=0,1,\ldots,n-1,\label{y77}\\
	&A_{j+1}^{*}\subset \tilde A_{j+1}^{*}\subset A_{j}^{*},\quad j=0,1,\ldots,n-1.
\end{align}
By a4) we have $\tilde G\setminus F_A\neq \varnothing $ whenever $G\in \ch(A)$. So we can fix  points $\xi_j\in \tilde A_{j+1}\setminus F_{A_{j}},\quad j=0,1,\ldots,n-1$. 
Applying T1) and T2) properties (see also \rem{R1}) and the definition of $F_A$ in \e{y75}, we get inequalities
\begin{align}
&\left\|T(f\cdot \ZI_{A_j^{**}\setminus \tilde A_{j+1}^{**}})(x)-T(f\cdot \ZI_{A_j^{**}\setminus \tilde A_{j+1}^{**}})(\xi_j)\right\|_\ZV\\
&\qquad\le \ZL_1\langle f\cdot \ZI_{A_j^{**}\setminus \tilde A_{j+1}^{**}}\rangle_{{\tilde A_{j+1}^{*}}}^*\le \ZL \MM(f\cdot \ZI_{A_{j}^{**}})(\xi_j)\\
&\qquad \le \Gamma(f\cdot \ZI_{A_{j}^{**}})(\xi_j)\lesssim \ZL  \langle f\rangle_{A_j^{**}},\label{a75}
\end{align}
\begin{align}
	\left\|T(f\cdot \ZI_{A_j^{**}\setminus \tilde A_{j+1}^{**}})(\xi_j)\right\|_\ZV\le T^{*}(f\cdot \ZI_{A_j^{**}})(\xi_j)\le \ZL  \langle f\rangle_{A_{j}^{**}},\label{a105}
\end{align}
\begin{align}
	\left\|T(f\cdot \ZI_{\tilde A_{j+1}^{**}\setminus A_{j+1}^{**}})(x)\right\|_\ZV\le \ZL \langle f\rangle_{\tilde A_{j+1}^{**}}\le \ZL \MM(f\cdot \ZI_{A_{j}^{**}})(\xi_j)\lesssim \ZL  \langle f\rangle_{A_j^{**}},\label{a76}
\end{align}
where in \e{a76} we also use \e{y76}. Thus, applying  \e {a75}, \e {a105} and  \e {a76}, we conclude	
\begin{align}
	\left\|T(f\cdot \ZI_{A_j^{**}})(x)\right\|_\ZV	&\le\left\|T(f\cdot \ZI_{A_j^{**}\setminus \tilde A_{j+1}^{**}})(x)\right\|_\ZV+\left\|T(f\cdot \ZI_{\tilde A_{j+1}^{**}})(x)\right\|_\ZV\\
	&\le \left\|T(f\cdot \ZI_{A_j^{**}\setminus \tilde A_{j+1}^{**}})(x)-T(f\cdot \ZI_{A_j^{**}\setminus \tilde A_{j+1}^{**}})(\xi_j)\right\|_\ZV\\
	&\qquad+\left\|T(f\cdot \ZI_{A_j^{**}\setminus \tilde A_{j+1}^{**}})(\xi_j)\right\|_\ZV\\
	&\qquad+\left\|T(f\cdot \ZI_{\tilde A_{j+1}^{**}\setminus A_{j+1}^{**}})(x)\right\|_\ZV+\left\|T(f\cdot \ZI_{A_{j+1}^{**}})(x)\right\|_\ZV\\
	&\le C \ZL\cdot\langle f\rangle_{A_j^{**}}+\left\|T(f\cdot \ZI_{A_{j+1}^{**}})(x)\right\|_\ZV\label{y24}
\end{align}
with an admissible constant $C>0$. Since $x\in A_n\setminus  F_{A_n}$, we can  also write $\left\|T(f\cdot \ZI_{A_{n}^{**}})(x)\right\|_\ZV\le \ZL  \langle f\rangle_{A_n^{**}}$. Finally, using \e{y77} and the iteration of \e{y24}, we conclude
\begin{align}
		\|Tf(x)\|_\ZV=\|T(f\cdot \ZI_{A_0^{**}})(x)\|_\ZV\lesssim \ZL\sum_{j=0}^n \langle f\rangle_{A_j^{**}}\lesssim \ZL\sum_{A\in \ZS} \langle f\rangle_{[[A]]}\ZI_{\bar A}(x).
\end{align}
This gives the required sparse domination \e{y51} with respect to the sparse family of balls $\ZA=\{[[A]]:\, A\in \ZS\}$ and the martingale-collection $\bar \ZA=\{A:\, A\in \ZS\}$. To show \e{y53} we use \e{s27} that holds for the initial family $\ZS$.  Indeed, from \e{s27} and the definition of martingale-collection it follows that
\begin{align}
	&\sum_{G\in \ch^{n}(A_0)}\mu(\bar G)\lesssim \alpha^n \mu(A_0),\\
	&\bigcup_{G\in \ch^{n+1}(A_0)}\bar G\subset \bigcup_{G\in \ch^{n}(A_0)}\bar G
\end{align}
that easily implies \e{y53} choosing a small enough $\alpha$. Condition \e{y52} holds with a ball $B'=[[[A_0]]]$. Applying \trm{T6}, we derive the sparse domination \e{y51} for the truncated operator $T^*$ too. Theorem is proved.
\end{proof}

\section{An abstract version of an inequality of Lerner}\label{S6}
Denote the $\alpha$-oscillation of a measurable function $f\in L^0(X,\ZU)$ on a ball $B\in\ZB$ by
\begin{equation}\label{y100}
\OSC_{B,\alpha}(f)=\inf_{E\subset B:\, \mu(E)> \alpha\mu(B)}\OSC_E(f),\quad 0<\alpha <1 \text{ (see \e{y97})}.
\end{equation}
\begin{definition}\label{D2}
	For $f\in L^0(X,\ZU)$ and a ball $B\in\ZB$ let $M_f(B)\subset B$ be the union of all measurable sets $E\subset B$, satisfying 
	\begin{equation}\label{y99}
		\mu(E)>\mu(B)/2,\quad \OSC_E(f)\le 2\OSC_{B,1/2}(f).
	\end{equation}
A vector $m=m_f(B)\in \ZU$ is said to be a median of the function $f$ on $B$ if there exists a point $x\in M_f(B)$ such that $f(x)=m$.
\end{definition}
It is clear $M_f(B)\neq \varnothing$ and so the set of medians of $f$ on any ball $B$ is non-empty. Observe that for any two points $x,x'\in M_f(B)$ we have $x\in E$ and $x'\in E'$, where both $E$ and $E'$ satisfy the conditions in \e{y99} and so there exists a point $y\in E\cap E'$. Hence we can write 
\begin{align}
	\|f(x)-f(x')\|_\ZU&\le \|f(x)-f(y)\|_\ZU+\|f(y)-f(x')\|_\ZU\\
	&\le \OSC_E(f)+\OSC_{E'}(f)\le 4\OSC_{B,1/2}(f).
\end{align}
Thus we have the following.
\begin{lemma} \label{L7}
	For any function $f\in L^0(X,\ZU)$ and a ball $B\in\ZB$ it holds the inequality 
\begin{equation}
\OSC_{M_f(B)}(f)\le 4\OSC_{B,1/2}(f).
\end{equation}
\end{lemma}
The following proposition is a version of a similar bound due to Lerner \cite{Ler1}, where  real valued functions on $\ZR^n$ were considered. Note that in \cite{Ler1} so called local mean oscillation instead of $\alpha$-oscillation \e{y100} was used, which definition uses a non-increasing rearrangement of the function $f$ and is not applicable for vector-valued functions. On the other hand in the case of real-valued functions both definitions of oscillations are equivalent. Similarly, the definitions of medians here and in \cite{Ler1} are also different. Finally, note that the definition of the local mean oscillation and the median for the real-valued functions was introduced by Jawerth and Torchinsky in \cite{JaTo}. 
\begin{proposition}\label{T3}
	Let $\ZB$ be a doubling ball-basis. Then for any measurable function $f\in L^0(X,\ZU)$ and a ball $A_0$ there exist a sparse family of balls $\ZA$ and a martingale-family of measurable sets $\bar \ZA\Subset \ZA$ such that
	\begin{align}
		& \bigcup_{G\in \ZA}G\subset A_0^*,\label{y89}\\
		&\mu\big\{x\in X:\,\sum_{G\in \ZA}\ZI_{\bar G}(x)>\lambda\big\}\lesssim \exp(-c\lambda)\mu(A_0),\quad \lambda>0,\label{y90}\\
		&\|f(x)- m_f(A_0)\|_\ZU\lesssim \sum_{G\in \ZA}\OSC_{G,\beta}(f)\ZI_{\bar G}(x)\text{ for a.e. }x\in A_0,\label{y91}
		\end{align}
	where $0<\beta<1$, $c>0$ are admissible constants and $m=m_f(A_0)$ is a median of a function $f$ over the ball $A_0$.
\end{proposition}
\begin{proof}
	Given $0<\beta<1$.  For any ball $A\in \ZB$ we may fix a set $E_{A}\subset [A]$ such that
		\begin{equation}\label{a13}
			\mu(E_A)\ge \beta\mu([A]),\quad \OSC_{E_A}(f)\le 2\OSC_{[A],\beta}(f).
		\end{equation}
	For the sets $F_A=[A]\setminus E_A$ we have $\mu(F_A)\le\alpha \mu([A])\lesssim\alpha \mu(A)$, where $\alpha=1-\beta$. Thus, applying \pro{P11}, we find a sparse tree-collection of balls $\ZS$, satisfying a1)-a4). Since our ball-basis is doubling, we may also have \e{y27} for $G'=G$. Besides, for a small enough $\alpha$ inequality \e{s27} may imply $\mu([G])<\mu(A)$ and so we get $[G]\subset A^*\subset [A]$ whenever $G\in \ch(A)$. Hence we can write
	\begin{align}
		&G\in \ch(A)\Rightarrow [G]\subset A^*\subset [A],\label{y45}\\
		&G\in \ch(A)\Rightarrow\mu(G\cap F_A)<\mu(G)/2 \text { (see \e{y27})}.\label{y46}
	\end{align}
Moreover, for an admissible $\alpha$, inequality \e{s27} also implies
\begin{equation}\label{y88}
	\sum_{G\in \ch(A)}\mu([G])\le \mu([A])/2.
\end{equation}
According to a2) we have
\begin{equation}\label{y48}
A_0\subset_{a.e.} \bigcup_{A\in \ZS}A\cap E_{A}\subset \bigcup_{A\in \ZS} E_{A}.
\end{equation}
	 We also claim for a smaller enough $\alpha$ that 
	\begin{equation}\label{a17}
		G\in \ch(A)\Rightarrow E_A\cap E_G\neq\varnothing.
	\end{equation} 
	Indeed, by \e{y45} and \e{y46} we have $\mu(G\cap E_A)>\mu(G)/2$.
	Thus, using also \e{a13}, we get
	\begin{align}
		\mu(E_A\cap E_G)&\ge \mu((E_A\cap G)\cap (E_G\cap G))\\
		&=\mu(E_A\cap G)-\mu((E_A\cap G)\setminus (E_G\cap G))\\
		&=\mu(E_A\cap G)-\mu(G\setminus (E_G\cap G))\\
		&\ge \mu(G)/2-\mu([G]\setminus E_G)\\
		&\ge \mu(G)/2-\alpha\mu([G])\\
		&\ge \mu(G)(1/2-\alpha\ZK)>0,
	\end{align}
	and so \e{a17} follows. 	
	According to \e{y45} the tree-collection of measurable sets $\ZA=\{[A]:\, A\in \ZS\}$ satisfies \e{y44}. Thus, applying \pro{L13} and \e{y48}, we find a martingale family of measurable sets $\bar \ZA\Subset \ZA$ such that
	\begin{equation}
		\bigcup_{A\in \ZS}\bar A\cap E_A=\bigcup_{A\in \ZS}E_A\supset _{a.e.} A_0 \text{ (see \e{y64})}.
	\end{equation}
	Thus for a.e. $x\in A_0$ we have $x\in \bar A \cap E_A$ for some $A\in \ZS$. Fix a unique chain sequence of balls $A_0, A_1,\ldots ,A_n=A$ such that  $A_{j+1}\in \ch(A_j)$, $j=0,1,\ldots,n-1$. According to the martingale property of $\bar \ZA$, the relation $x\in \bar A=\bar A_n$ implies 
	\begin{equation}\label{y35}
		x\in \bigcap_{j=0}^n\bar A_j.
	\end{equation}
	Using \e{a17}, we find points $\xi_j\in E_{A_{j-1}}\cap E_{A_{j}}$, $j=1,2,\ldots, n-1$. Set also $\xi_n=x$. Since $\xi_{j},\xi_{j+1}\in E_{A_j}$, from \e{a13} it follows that
	\begin{equation}\label{a28}
		\|f(\xi_{j})-f(\xi_{j+1})\|_\ZU\le 2\OSC_{[A_j],\beta}(f),\quad j=1,2,\ldots, n-1.
	\end{equation} 
	In addition, if $\beta>1/2$ from \e{a13}  it follows that $E_{A_0}\subset M_f(A_0)$(see \df{D2}). Then, since $\xi_1\in E_{A_0}$, by \lem{L7} we can write
	\begin{equation}\label{a49}
		\|f(\xi_{1})-m_f(A_0)\|_\ZU\le 4\OSC_{[A_0],1/2}(f)\le 4\OSC_{[A_0],\beta}(f).
	\end{equation} 
Applying, \e{y35}, \e{a28} and \e{a49}, we obtain
	\begin{align}
		\|f(x)-m_f(A_0)\|_\ZU&=\|f(\xi_n)-m_f(A_0)\|_\ZU\\
		&\le \|f(\xi_1)-m_f(A_0)\|_\ZU+\sum_{j=1}^{n-1}\|f(\xi_j)-f(\xi_{j+1})\|_\ZU\\
		&\le 4\sum_{j=0}^{n-1} \OSC_{[A_j],\beta}(f)\ZI_{\bar A_j}(x)\\
		&\le 4\sum_{A\in \ZA} \OSC_{A,\beta}(f)\ZI_{\bar A}(x),
	\end{align}
which implies \e{y91}. Relation \e{y89} immediately follows from \e{y45} and \e{y88} implies \e{y90}. This completes the proof.
\end{proof}

\section{Proofs of Theorems \ref{T5} and \ref{T8} ($\varrho=\rho=1/r$, $r\ge 1$)}
We will need the following two standard lemmas.
\begin{lemma}\label{L6}
	For any function $f\in L^1_{loc}(X,\ZU)$ and balls $A\subset B$ we have
	\begin{equation*}
		\|f_A-f_B\|_\ZU\le \left(\frac{\mu(B)}{\mu(A)} \right)^{1/r}\langle f\rangle_{\#,B}.
	\end{equation*}
	\begin{proof}
		We have
		\begin{align}
			\left\|f_{A}-f_{B}\right\|_\ZU&\le \frac{1}{\mu(A)}\int_{A}\|f-f_{B}\|_\ZU\\
			&\le\left(\frac{\mu(B)}{\mu(A)} \right)^{1/r}\left(\frac{1}{\mu(B)}\int_{B}\|f-f_{B}\|_\ZU^r\right)^{1/r}\le \left(\frac{\mu(B)}{\mu(A)} \right)^{1/r}\langle f\rangle_{\#,B}.
		\end{align}
	\end{proof}
\end{lemma}
\begin{lemma}\label{L10}
	Let the ball-basis $\ZB$ satisfy the doubling condition. Then for any $f\in L^1_{loc}(X,\ZU)$ and balls $A\subset B$ it holds the inequality
	\begin{equation*}
		\left\|f_A-f_{B}\right\|_\ZU\lesssim \log(1+\mu(B)/\mu(A))\cdot \langle f\rangle_{\#,A}^*.
	\end{equation*}
\end{lemma}
\begin{proof} 
	According to \lem {L1}, we may find a sequence of balls $A=A_0\subset A_1\subset \ldots\subset A_n=[B]$ satisfying such that $\mu(A_{k+1})\lesssim \mu(A_k)$. One can easily check that $n\lesssim \log \left(1+\mu(B)/\mu(A)\right)$.
By \lem {L6} we have
	\begin{equation}
		\left\|f_{A_k}-f_{A_{k+1}}\right\|_\ZU\lesssim \langle f\rangle_{A_{k+1},\#}\le \langle f\rangle_{\#,A}^*.
	\end{equation}
	Thus we obtain
	\begin{align*}
		\|f_A-f_B\|_\ZU&\le \sum_{k=0}^{n-1}\|f_{A_k}-f_{A_{k+1}}\|_\ZU+\|f_{B}-f_{[B]}\|_\ZU\lesssim n\langle f\rangle_{\#,A}^*\\
		&\lesssim \log(1+\mu(B)/\mu(A))\langle f\rangle_{\#,A}^*
	\end{align*}
\end{proof}
\begin{proof}[Proof of \trm{T5}]
	Let $f\in L^r(X,\ZU)$ and $B\in \ZB$. Applying restriction condition R5) we find a ball $B'\supset B$ such that 
	\begin{equation}
	\sup_\alpha \OSC_B(T_\alpha(\ZI_{B''}))<\delta
	\end{equation}
for every ball $B''\supset B'$. Then let $G_n$ be the sequence of balls generated from \lem{L5}. We have $\langle f\rangle_{G_n}^*\to 0$ and so there is a ball $G=G_n\supset B'^*$ such that $\langle f\rangle^*_{G} <\delta$.
Applying restriction condition R5), we can write
	\begin{align}
	&B^*\subset G,\\
	&\sup_\alpha \OSC_B(T_\alpha(\ZI_{G}))<\delta,\label{y86}\\
	&\sup_\alpha\OSC_{B}(T_\alpha(f\cdot \ZI_{X\setminus G}))\le \sup_\alpha \ZL_1(T_\alpha)\cdot \langle f\rangle^*_{G}<\delta \sup_\alpha\ZL_1(T_\alpha). \label{y85}
	\end{align}
	Then we can write 
		\begin{align*}
		f&=f\cdot \ZI_{X\setminus G}+(f-f_B)\ZI_{B^*}+(f-f_B)\ZI_{G\setminus B^*}+f_B\ZI_{G}\\
		&=f_0+f_1+f_2+f_3.
	\end{align*}
From \e{y86}, \e{y85} we can obtain
\begin{equation}\label{x2}
	\sup_\alpha \OSC_B(T_\alpha(f_0))<\varepsilon,\quad \sup_\alpha \OSC_B(T_\alpha(f_3))<\varepsilon
\end{equation}
with an arbitrary small $\varepsilon>0$. Consider the set $E_{B}=\{y\in B:\, Tf_1(y)\le \lambda\}$. According to the $\ZL_0$-condition of the maximal operator $T$ for 
	 \begin{align}
	 	\lambda=\ZK^{1/r}(1-\beta)^{-1/r}\ZL_0(T)\langle f_1\rangle_{B^*}
	 \end{align}
	 we have
	\begin{equation}
		\mu(B\setminus E_{B})=\mu\{y\in B:\, Tf_1(y)>\lambda\}\le\frac{(1-\beta )\mu(B^*)}{\ZK}\le  (1-\beta )\mu(B).
	\end{equation}
	Thus, applying also \lem{L6}, we can write
	\begin{align}
		&\mu(E_{B})>\beta\mu(B),\\
		& Tf_1(y)\lesssim(1-\beta)^{-1/r}\ZL_0(T)\langle f-f_B\rangle_{B^*}\\
		&\qquad \quad\le (1-\beta)^{-1/r}\ZL_0(T)(\langle f-f_{B^*}\rangle_{B^*}+\|f_B-f_{B^*}\|_\ZU)\\
		&\qquad \quad\lesssim(1-\beta)^{-1/r} \ZL_0(T)\langle f\rangle^*_{\#,B}\text{ for every }y\in E_B.\label{y87}
	\end{align}
	Now choose two arbitrary points $x,x'\in E_B$ and suppose $Tf(x)\ge Tf(x')$. Then for some $\alpha$ we have
	\begin{align}
		|Tf(x)-Tf(x')|=Tf(x)-Tf(x')\le 2( \|T_\alpha f(x)\|_\ZU-\|T_\alpha f(x')\|_\ZV).
	\end{align}
	On the other hand, using \e{y87}, for small enough $\varepsilon$ in \e{x2} we obtain
	\begin{align}
		\|T_\alpha f(x)\|_\ZV-&\|T_\alpha f(x')\|_\ZV\\
		&\le \|T_\alpha f(x)-T_\alpha f(x')\|_\ZV\\
		&\le \OSC_B(T_\alpha(f_0))+\OSC_B(T_\alpha(f_3))\\
		&\qquad +\|T_\alpha f_1(x)\|_\ZV+\|T_\alpha f_1(x')\|_\ZV+\|T_\alpha f_2(x)-T_\alpha f_2(x')\|_\ZV\\
		&\le 2\varepsilon +\|T f_1(x)\|_\ZV+\|T f_1(x')\|_\ZV+\|T_\alpha f_2(x)-T_\alpha f_2(x')\|_\ZV\\
		&\lesssim (1-\beta)^{-1/r}\ZL_0(T)\langle f\rangle^*_{\#,B}\\
		&\qquad +\|T_\alpha f_2(x)-T_\alpha f_2(x')\|_\ZV.\label{y28}
	\end{align}
	Applying T1)-condition, we get
	\begin{align}
	\|T_\alpha f_2(x)-T_\alpha f_2(x')\|_\ZV&\le \OSC_{B}(T_\alpha (f_2))\\
	&\le \ZL_1(T_\alpha)\sup_{C\in \ZB\,C\supset B}\log^{-1}\left(1+\frac{\mu(C)}{\mu(B)}\right)\langle f_2\rangle_{C}\\
	&\le 2\ZL_1(T_\alpha)\log^{-1}\left(1+\frac{\mu(A)}{\mu(B)}\right)\langle f_2\rangle_{A}\label{x10}
	\end{align} 
	for some ball $A\supset B$.
	Clearly we can suppose that $\mu(A)>2\mu(B)$, since otherwise we would have $A\subset B^*$ that means $\langle f_2\rangle_{A,r}=0$ and so the left hand side of \e {x10} is zero. Thus, applying \lem{L10}, we obtain
	\begin{align}
		\|T_\alpha f_2(x)-T_\alpha f_2(x')\|_\ZV&\le 2\ZL_1(T_\alpha)\log^{-1}\left(1+\frac{\mu(A)}{\mu(B)}\right)\langle f-f_B\rangle_{A}\\
		&\le 2\ZL_1(T_\alpha)\log^{-1}\left(1+\frac{\mu(A)}{\mu(B)}\right)(\langle f-f_A\rangle_{A} +|f_A-f_B|)\\
		&\lesssim \ZL_1(T_\alpha)\log^{-1}\left(1+\frac{\mu(A)}{\mu(B)}\right)\log\left(1+\frac{\mu(A)}{\mu(B)}\right)\langle f\rangle^*_{\#,B}\\
		&=\ZL_1(T_\alpha)\langle f\rangle^*_{\#,B}.\label{y29}
	\end{align}
	Thus, combining \e{y28} and \e{y29}, we obtain \e{y30}.
\end{proof}
\begin{proof}[Proof of \trm{T8}]
	Let $f\in L^r(X,\ZU)$ and $B$ be a ball.  Applying \pro{T3} to the function $T(f)$ with $A_0=B$, we find sparse and martingale families of balls $\bar\ZA \Subset \ZA$, satisfying \e{y89}, \e{y90} and the bound
	\begin{equation}\label{y92}
		|Tf(x)- m_{T(f)}(B)|\lesssim \sum_{G\in \ZA}\OSC_{G,\beta}(T(f))\ZI_{\bar G}(x)\text{ for a.e. }x\in B
	\end{equation}
	Combining  \e{y92} with the oscillation estimate in \e{y30}, we get \e{y94}. Then \e{y89} $\Rightarrow$ \e{y52} and \e{y90} $\Rightarrow$ \e{y53}. 
\end{proof}
\section{Exponential inequalities}
In this section we provide applications of Theorems \ref{T7} and \ref{T8} in exponential decay inequalities as a special types of good-lambda inequalities.
First good-lambda inequalities arreared in the 1970's and till are significant tools in the study of operators in harmonic analysis. Originally the good-lambda technique was introduced by Burkholder and Gundy in \cite{Bur},\cite{BuGu}, where the method was applied to operators generated by martingales. Later, Coifman and Fefferman \cite{CoFe} proved a  good-lambda inequality involving the truncated Calder\'on-Zygmund operator $T^*$ and the Hardy-Littlewood maximal function on $\ZR^n$. Namely they established
\begin{equation}\label{I1}
	w\{T^*(f)>2\lambda,\, M(f)<\beta \lambda\}|\le c(\beta)w\{T^*(f)>\lambda\},\quad \lambda>0,
\end{equation}
where $w$ is an arbitrary $A_\infty$ measure and $c(\beta)\to 0$ as $\beta\to 0$. This inequality provides a control of $T^*$ by mean of the maximal function $M$, namely one can easily deduce the bound $	\|T^*f\|_{L^p(w)}\le c\|Mf\|_{L^p(w)}$.
The classical proof of \e{I1} is based on a local type good-lambda inequality 
\begin{equation}\label{I3}
	|\{x\in B:\, T^*f(x)>2\lambda,\, Mf(x)<\beta \lambda\}|\le c\beta^\delta |B|,\quad \lambda>0,
\end{equation}
where $B$ is an Euclidean cube in $\ZR^n$. In fact, to deduce \e{I1} from \e{I3} one just need to apply \e{I3} to the cubes of the Whitney decomposition of the set $\{T^*(f)>\lambda \}$. Buckley  in \cite{Buck} improved the constant on the right-hand side of \e{I3} replaced it by $c_1\exp(-c_2/\beta)$. Buckley proved this estimate using as a model a more classical inequality due to Hunt \cite{Hunt} for the conjugate function. Buckley's proof is based on an approach given  \cite{Hunt}. Both proofs use the Calder\'on-Zygmund standard $\lambda$-decomposition of the function to a sum of a bounded function $g$ and a function $h$ supported on the intervals outside of the set $\{Mf(x)<\beta \lambda\}\}$. The exponential decay of $\{T(g)>\lambda\} $ is a classical inequality. As for $T(h)$ it has an exponential distribution on the set $\{Mf(x)<\beta \lambda\}\}$, since $T(h)$ admits a pointwise estimate by the $\Delta$-function used in Carleson's celebrated paper (\cite{Car}, Lemma 5). Improving Buckley's inequality, Karagulyan \cite{Kar2} established a new type of exponential inequality
\begin{equation}\label{I4}
	|\{x\in B:\, T^*f(x)>\lambda Mf(x)\}|\le c_1\exp(-c_2\lambda)|B|,
\end{equation}
which proof uses rather different approach than papers \cite{Buck, CoFe} provide. It applies a careful covering of level sets of $T^*(f)/M(f)$ by exponentially sparsening collections of balls. Furthermore such exponential type inequalities involving other operators (Littlewood-Paley square functions, multilinear Calder\'on-Zygmund operators and commutators) were obtained by Ortiz-Caraballo, P\'{e}rez and Rela \cite{Per}. In this context one can also consider the recent paper of Canto and P\'{e}rez \cite{CaPe}, where authors give two interesting extensions of John-Nirenberg theorem in weighted setting. 

Karagulyan in \cite{Kar1} gave a general approach to good-$\lambda$ inequalities, providing domination conditions, which imply good-$\lambda$ inequalities for couples of measurable functions. As an immediate corollary of \pro{T3}, we recover one of two main inequalities of paper \cite{Kar1} for vector-valued functions. 
\begin{definition}
	Let $f,g\in L^0(X,\ZU)$ be measurable functions.  The function  $f$ is said to be strongly dominated by $g$ if for any $0<\alpha<1$ there exists a number $\beta=c(\alpha)>0$ such that 
	\begin{equation}\label{01}
		\OSC_{B,\alpha}(f)< \beta\cdot\INF_{B}(g).
	\end{equation}
\end{definition}
\begin{corollary}
	If the ball-basis $\ZB$ in a measure space $(X,\mu)$ is doubling and measurable functions $f,g\in L^0(X,\ZU)$  satisfy strong domination condition \e{01}, then for any ball $B\in \ZB$ we have
	\begin{equation}\label{6}
		\mu\{x\in B:\,\|f(x)-m_f(B)\|_\ZU>\lambda  \|g(x)\|_\ZU\}\lesssim\exp(-c\cdot\lambda )\mu(B),\quad \lambda>0,
	\end{equation}
	where $c>0$ is an admissible constant and $m_f(B)$ is a median of $f$ on the ball $B$.
\end{corollary}
The following two inequalities are immediate consequences of Theorems \ref{T7} and \ref{T8}. Note that the first one is for $\BO$ operators in most general settings, while in the second one our ball-basis is doubling and we have the restricted version of $\BO$ operators. Hence, given a function $f\in L^r(X)$ and a ball $B\in \ZB$. Applying sparse bound \e{y51} of \trm{T7}, we can write
\begin{equation}\label{x1}
	\|Tf(x)\|_\ZV\lesssim \MM f(x)\sum_{G\in \ZA}\ZI_{\bar G}(x),\quad x\in B,
\end{equation}
then, using also \e{y53}, we immediately arrive to the following exponential estimate.
\begin{corollary}\label{C8}
	Let $(X,\mu)$ be measure space equipped with a ball-basis $\ZB$ (which can also be non-doubling) and $T$ be a general $\BO$-operator. Then the inequality
	\begin{equation}\label{a44}
		\mu\{x\in B:\,\|Tf(x)\|_\ZV>t \cdot \MM f(x)|\}\lesssim \exp (-t/c(T))\mu(B),\quad t>0,
	\end{equation}
	holds for any function $f\in L^r(X,\ZU)$ and a ball $B$, where $c(T)\sim \ZL_0(T)+\ZL_1(T)+\ZL_2(T)$.
\end{corollary}
Now let 
\begin{equation}
	\MM_\# f(x)=\sup_{B\in \ZB,\, B\ni x} \langle f\rangle_{\#,B}=\sup_{B\in \ZB,\, B\ni x} \langle f\rangle_{\#,B}^*
\end{equation}
denote the sharp maximal function.
Similarly to \e{x1} from \trm{T8} it follows that 
\begin{equation}
	\|Tf(x)-m_{T(f)}(B)\|_\ZV\lesssim \MM_\# f(x)\sum_{G\in \ZA}\ZI_{\bar G}(x),\quad x\in B,
\end{equation}
for any $f\in L^r(X)$ and a ball $B\in \ZB$. Then, using also \e{y52}, we obtain
\begin{corollary}\label{C5}
	Under the assumptions of \trm{T8} for any $f\in L^r(X,\ZU)$ and a ball $B$ the maximal operator $Tf=\sup_\alpha|T_\alpha f|$
	satisfies
	\begin{equation}\label{x3}
		\mu\{x\in B:\,|Tf(x)-m_{T(f)}(B)|>t  \MM_\# f(x)|\}\lesssim \exp (-ct)\mu(B),\, t>0.
	\end{equation}
\end{corollary}
Let denote by $\BMO(X,\ZU)$ the space of functions $f\in  L^0(X,\ZU)$ such that
\begin{equation}
	\|f\|_\BMO=	\sup_{B\in \ZB}\frac{1}{\mu(B)}\int_{B}|f-f_B|<\infty.
\end{equation}
\begin{corollary}\label{C4}
	Under the assumptions of \trm{T8} the maximal operator $Tf=\sup_\alpha|T_\alpha f|$ is bounded on $\BMO$. More precisely, 
\begin{equation}
\|T(f)\|_{\BMO(X,\ZR)}\lesssim \|f\|_{\BMO(X,\ZU)}
\end{equation}
for any function
\begin{equation}
	f\in L^r(X,\ZU)\cap \BMO(X,\ZU).
\end{equation} 
\end{corollary}
\begin{proof}
	First apply \cor{C5} for the identical operator $T(f)=f$, which is obviously a restricted type of $\BO$ operator with parameter $r=1$. We have
	$\MM_\#f(x)\le \|f\|_{\BMO}$ for all $x\in X$, where we suppose $r=1$ in the definition of $\MM_\#$. Thus, applying \e{x3}, we obtain the John-Nirenberg inequality
	\begin{equation}\label{x4}
		\mu\{x\in B:\,\|f(x)-m\|_{\ZU}>t \cdot\|f\|_\BMO\}\lesssim \exp (-c\cdot t)\mu(B),\quad t>0,
	\end{equation}
where by a standard argument  $m$ can be also replaced by $f_B$. Returning to the general case of $r\ge 1$ from \e{x4} it follows that 
\begin{equation}\label{x5}
	\MM_\#f(x)\lesssim \|f\|_\BMO,\quad x\in X.
\end{equation}
Now let $T$ be the operator in \cor{C5}. From \e{x3} and \e{x5} we obtain
\begin{equation}
	\mu\{x\in B:\,\|Tf(x)-m_{T(f)}(B)\|_\ZV>t \cdot\|f\|_\BMO\}\lesssim \exp (-c\cdot t)\mu(B),\quad t>0,
\end{equation}
then by the same standard argument we get $\|T(f)\|_\BMO\lesssim \|f\|_\BMO$.
\end{proof}
Finally observe that $\BO$ operators satisfying a strong-sublinearity condition
\begin{equation}\label{y41}
	\|Tf(x)-Tg(x)\|_\ZV\le \|T(f-g)(x)\|_\ZV
\end{equation}
boundedly map $L^\infty$ to $\BMO$. Clearly, \e {y41} implies the sublinearity of $T$ and all linear operators as well as positive sublinear operators are strong-sublinear. Hence  the following statement is true.
\begin{theorem}\label{T1}
	Every strong-sublinear $\BO$ operator $T$ boundedly maps $L^\infty$ into $\BMO$. Moreover,
	\begin{equation*}
		\|T(f)\|_{\BMO(X,\ZV)}\lesssim (\ZL_0(T)+\ZL_1(T)+\ZL_2(T))\|f\|_{L^\infty(X,\ZU)}.
	\end{equation*}
\end{theorem}
\begin{proof}
	Sparse domination \e{y51} implies boundedness of $T$ on $L^p$ for any $r<p<\infty$.  Namely we can write
	\begin{equation}\label{x15}
		\|T\|_{L^p(X,\ZU)\to L^p(X,\ZV)}\lesssim (\ZL_0(T)+\ZL_1(T)+\ZL_2(T)).
	\end{equation}
	Take a function $f\in L^\infty (X,\ZU)$ and for $B\in \ZB$ denote
	\begin{equation}\label{x8}
		\alpha_B=\left(T(f\cdot \ZI_{X\setminus B^*})\right)_B.
	\end{equation}
	Strong-subadditivity of $T$ implies
	\begin{equation*}
		\|Tf-T(f\cdot \ZI_{X\setminus B^*})\|_\ZV\le \|T(f\cdot \ZI_{B^*})\|_\ZV.
	\end{equation*}
	Hence, applying T1)-condition and \e{x15}, we get
	\begin{align*}
		\langle Tf-\alpha_B\rangle_{B}&=\left\langle Tf-T(f\cdot \ZI_{X\setminus B^*})+T(f\cdot \ZI_{X\setminus B^*})-\alpha_B\right\rangle_{B}\\
		&\le \left\langle T(f\cdot \ZI_{B^*})\right\rangle_{B}+\OSC_{B}(T(f\cdot \ZI_{X\setminus B^*}))\\
		&\lesssim \left\langle T(f\cdot \ZI_{B^*} \right\rangle_{B}+\ZL_1(T)\langle f\rangle_B^*\\
		&\lesssim \left(\frac{1}{\mu(B)}\int_X\|T(f\cdot \ZI_{B^*}\|_\ZV^p\right)^{1/p}+\ZL_1(T)\|f\|_\infty\\
		&\le \|T\|_{L^p\to L^p}\left(\frac{1}{\mu(B)}\int_X\|f\cdot \ZI_{B^*}\|_\ZU^p\right)^{1/p}+\ZL_1(T)\|f\|_\infty\\
		&\le (\ZL_0(T)+\ZL_1(T)+\ZL_2(T))\|f\|_\infty.
	\end{align*}
\end{proof}
\noindent
\textbf{Remarks:}
\begin{enumerate}
	\item In the case when $T$ is the identical operator on $L^p(\ZR^d)$, i.e. $T(f)=f$, inequality \e{x3} was proved in \cite{Kar2}. Also note that if in addition $f\in \BMO(\ZR^d)$, then it gives the John-Nirenberg classical theorem.
	\item When $T$ is a general $\BO$ operator (on doubling ball bases)  inequality \e{x3} with the usual maximal function $\MM $ instead of $\MM_\# $ was proved in \cite{Kar2} (see also  \cite{Kar1} for the case of classical a Calder\'on-Zygmund operators).

\end{enumerate}

\section{Examples of $\BO$ operators and corollaries}

\subsection{Maximally modulated Calder\'on-Zygmund operators} It was proved in \cite{Kar3} that Calder\'on-Zygmund operators on metric measure spaces of homogeneous type are $\BO$ operators. Here we consider Calder\'on-Zygmund operators, which are restricted type $\BO$ operators, namely satisfy R1)-R5). For the sake of simplicity we state our results on Euclidean spaces $\ZR^d$. A Calder\'on-Zygmund operator $T$ is a linear operator such that
\begin{equation}
	Tf(x)=\int_{\ZR^d}K(x,y)f(y)dy,\quad x\notin \supp f,
\end{equation}
for any compactly supported continuous function $f\in C(\ZR^d)$, where the kernel $K$ satisfies the conditions
\begin{align}
	&|K(x,y)|\le \frac{C}{|x-y|^d}\text{ if }x\neq y,\label{x24}\\
	&|K(x,y)-K(x',y)|\le \omega\left(\frac{|x-x'|}{|x-y|}\right)\cdot \frac{1}{|x-y|^d},\\
	&\text{ whenever } |x-y|>2|x-x'|.\label{x9}
\end{align}
Here $\omega:(0,1)\to (0,1)$ is non-decreasing and satisfies the Dini condition $\int_0^1\omega(t)/tdt<\infty$. Also it is supposed that $T$ can be  boundedly extended on $L^2(\ZR^d)$. Given function $g\in L^\infty(\ZR^d)$ consider the operator $T_g(f)=T(gf)$.
\begin{lemma}
	If a Calder\'on-Zygmund kernel satisfies logarithmic Dini condition
	\begin{equation}
		[\omega]=\int_0^1\frac{\omega(t)\log(1/t)}{t}dt<\infty,\label{x6}
	\end{equation} 
	then the operators $T_{g}$ satisfy \e{y83} uniformly with respect to functions $g\in L^\infty(\ZR^d)$, $\|g\|_\infty\le 1$.
\end{lemma}
\begin{proof}
	Let $f\in L^r(\ZR^d)$ and $B=B(x_0,R)\subset \ZR^d$ be an Euclidean ball of center $x_0$ and radius $R$. One can check that $B^*=B(x_0,R^*)$, where $R^*=(1+2\cdot 2^{1/d})R\in [R,5R]$. If any point $x\in B=B(x_0,R)$, using \e{x9} and \e{x6}, we get
	\begin{align}
		|T_{g}(f\cdot & \ZI_{\ZR^d\setminus B^*})(x)-T_{g}(f\cdot \ZI_{\ZR^d\setminus B^*})(x_0)|\\
		&\lesssim \int_{|y-x_0|>R^*}|K(x,y)-K(x_0,y)||f(y)|dy\\
		&\le \sum_{k=0}^\infty\int_{2^{k+1}R^*\ge |y-x_0|>2^kR^*}\omega\left(\frac{|x-x_0|}{|x_0-y|}\right)\cdot \frac{1}{|x_0-y|^d}|f(y)|dy\\
		&\lesssim  \sum_{k=1}^\infty k\omega(2^{-k})\frac{1}{k |B(x_0,2^{k+1}r)|}\int_{B(x_0,2^{k+1}r)}|f(y)|dy\\
		&\lesssim[\omega] \sup_{A\supset B}\log^{-1}\left(1+\frac{\mu(A)}{\mu(B)}\right)\langle f\rangle_{A},
	\end{align}
where $\sup$ is taken over all the balls $A\supset B$. Thus we get R4). 
\end{proof}
One  can similarly prove the following
\begin{lemma}
	If 
	\begin{equation}\label{x7}
		[K]=\sup_{ R,R':\, R'>2R}\,\sup_{t\in B(x,R)}\left|\int_{B(x,R')}K(x,y)-K(t,y)dy\right|<\infty,
	\end{equation} 
	then the  operators $T_{g}$, $\|g\|_\infty\le 1$, satisfy restriction condition R5) uniformly.
\end{lemma}
Papers \cite{Zor}, \cite{Lie} considered operator
\begin{equation}\label{x8}
	Tf(x)=\sup_{Q\in \ZQ_d}\left|\int_{\ZR^d}K(x,y)e^{2\pi i Q(y)}f(y)dy\right|
\end{equation}
where $K(x,y)$ is a H\"older continuous Calder\'on-Zygmund kernel and $\ZQ_d$ denotes the family of $d$-dimensional polynomials on $\ZR^d$. It was proved that operator \e{x8} is bounded on $L^p(\ZR^d)$ for every $1<p<\infty$. The H\"older continuous Calder\'on-Zygmund kernels satisfy \e{x9} with $w(t)=t^\delta$, $\delta>0$, and so we have \e{x6}. Thus, applying Theorems \ref{T5} and \ref{T8}, we arrive to the following
\begin{theorem}
Let $\{Q_\alpha\}\subset \ZQ_d$ be a family of polynomials and $K(x,y)$ be a Calder\'on-Zygmund kernel with modulus of continuity $w(t)=t^\delta$, $\delta>0$. If $K$ satisfies \e{x7}, then for the operator 
\begin{equation}\label{x14}
	Tf(x)=\sup_{\alpha}\left|\int_{\ZR^d}K(x,y)e^{2\pi i Q_\alpha(y)}f(y)dy\right|
\end{equation}
we have bounds \e{y94}, \e{y30} and \e{x3}.
\end{theorem}
\noindent
\textbf{Remarks:}
\begin{enumerate} 
	\item Examples of restricted type $\BO$ operators are Calder\'on-Zygmund classical convolution operators on $\ZR^d$, corresponding to kernels
	\begin{equation}
		K(x)=\frac{\Omega(x)}{|x|^d}
	\end{equation}
where $\Omega(x)$ satisfies
\begin{align*}
	&\Omega(\varepsilon x)=\Omega(x),\, \varepsilon>0,\quad \int_{S^{d-1}}\Omega(x)d\sigma=0,\\
	&\omega(\delta)=\sup_{\stackrel{|x-x'|\le \delta}{|x|=|x'|=1}}|\Omega(x)-\Omega(x')|,\quad \int_0^1\frac{\omega(t)\log(1/t)}{t}dt<\infty.
\end{align*}
(see \cite{Ste1}, chap. 2.4, for precise definition). Particular cases of such operators are Hilbert, Riesz and Beurling-Ahlfors transforms.
\item Significant cases of operators \e{x14} is the maximal partial sums operator of trigonometric Fourier series
\begin{equation}\label{x29}
	S^*f(x)=\sup_n|S_nf(x)|,
\end{equation}
and so it satisfies estimates \e{y94}, \e{x3} and is bounded on $\BMO(\ZT)$.  Hence we can write
\begin{equation}\label{x28}
	|\{x\in I:\, |S^*f(x)-m|>t \MM_\#f(x)\}|\lesssim |I|\exp (-ct).
\end{equation}
for any function $f\in L^r(\ZT)$, $r>1$, and an interval $I\subset \ZT$, where $m\in M_{S^*(f)}(I)$.
\item Inequality \e{x28} is true also for the Walsh maximal partial sums operator too, since each partial sum operator $S_n$ is a restricted type $\BO$ operator based on dyadic interval ball-basis of $[0,1]$. 
	\item Operators \e{x14} satisfy estimates \e{y94}, \e{x3} and are bounded on $\BMO$. Namely, for any $f\in L^r(\ZR^d)\cap \BMO(\ZR^d)$ we have $\|T(f)\|_\BMO\lesssim \|f\|_\BMO$. The fact that ordinary Calderón–Zyg\-mund operators boundedly map $L^\infty$ to $\BMO$ was independently obtained by Peetre \cite{Pee}, Spanne \cite{Spa} and Stein \cite {Ste}. Peetre \cite{Pee} also observed that translation-invariant Calderón–Zygmund operators actually map $\BMO$ to itself.  \cor{C4} and \trm{T1} provide extensions of these result to a wider class of $\BO$ operators. In particular, the maximal operators \e{x29} of both for trigonometric and Walsh systems are bounded on $\BMO$.
\end{enumerate}
\subsection{Martingales}
Let $\zB_n$ be a filtration defined in the introduction and $\zB=\cup_n\zB_n$ be the corresponding ball-basis. Given $f\in L^1(X,\ZU)$ denote the martingale difference sequence
\begin{equation}
	\Delta_Af(x)=\sum_{B:\, \pr(B)=A}\left(\frac{1}{\mu(B)}\int_Bf-\frac{1}{\mu(A)}\int_Af\right)\ZI_B(x)
\end{equation}
 We consider martingale transform and square function operators defined by
 \begin{align}
 	&M_\varepsilon f(x)=\sum_{A\in \zB}\varepsilon_{A}\Delta_A f(x),\quad \varepsilon_A=\pm 1,\label{x11}\\
 		&Sf(x)=\left(\sum_{A\in \zB}\left\|\Delta_A f(x)\right\|_\ZU^2\right)^{1/2}\label{x12}
 \end{align}
respectively.  It was proved in \cite{Kar3} that the real valued version of martingale transform $M_\varepsilon$ is a $\BO$ operator. Without any essential changes in the proof of this result one can obtain the same for the vector valued martingale transform and square function operators. Namely, 
\begin{theorem}\label{T13}
	Vector-valued versions of operators \e{x11} and \e{x12} are $\BO$ operators with parameters $\varrho=\rho=r=1$ and absolute constants $\ZL_0$, $\ZL_1=0$, $\ZL_2$. If the ball-basis (martingal filtration) is doubling, then both operators became restricted type $\BO$ operators.
\end{theorem}

\textbf{Remarks:}
\begin{enumerate}
	\item If $T$ is one of operators \e{x11} or \e{x12} based on general ball-basis $\zB$, then we have exponential estimate \e{a44}, where $\MM$ is the maximal function corresponding to $\zB$. In \cite{Kar1} the same inequality for \e{x11} was proved in the case of doubling ball-basis $\zB$ and real valued functions.
	\item If the ball-basis is doubling, then $T$ satisfies sparse bound \e{y84} and so exponential inequality \e{x3}. Exponential inequality \e{x3} for martingale transforms gives another extension of the mentioned inequality of \cite{Kar1}, replacing the maximal function by the sharp maximal function.
	\item So in the case of doubling ball-basis we have operators \e{x11} and \e{x12} are bounded on $\BMO(X,\ZU)$.
	\item Using the sparse domination \e{y51} we can write a sharp weighted norm estimate
	\begin{equation}
		\|T(f)\|_{L^p(X,w,\ZV)}\le C[w]_{A_p}\|f\|_{L^p(X,w,\ZU)}.
	\end{equation}
\end{enumerate}
for both operators \e{x11} and \e{x12}, where 
\begin{equation}
	[w]_{A_p}=\sup_{B\in \zB}\left(\frac{1}{\mu(B)}\int_Bw\right)\left(\frac{1}{\mu(B)}\int_Bw^{-1/(p-1)}\right)^{p-1}<\infty.
\end{equation}
is the Muckenhaupt characteristic of the weight $w$.
\subsection{Riesz potentials}
For $0<\alpha<n$ the fractional integral operator or Riesz potential $I_\alpha$ is defined by
\begin{equation}\label{x30}
	I_\alpha f(x)=\int_{\ZR^d}\frac{|f(y)|}{|x-y|^{n-\alpha}}dy.
\end{equation}
The related fractional maximal operator $\MM_\alpha$ is given by
\begin{equation}\label{x32}
	\MM_\alpha f(x)=\sup_{B\ni x}\frac{1}{|B|^{1-\alpha/n}}\int_B|f(y)|dy.
\end{equation}
\begin{theorem}
	Operators \e{x30} and \e{x32} are $\BO$ operators with parameters 
	\begin{equation}\label{x31}
		\varrho=r=1,\quad \rho=1-\alpha/n.
	\end{equation}
\end{theorem}
\begin{proof}
	The case of $\alpha=0$ for the maximal operator $\MM_\alpha$ was proved in \cite{Kar3} and the general case may be proved similarly. So let us consider the operator \e{x30}. We need to prove only conditions T0) and T1), since the Euclidean ball-basis is doubling. It is well-known that $I_\alpha$ is of weak $(1,q)$ type with $q$ satisfying $q^{-1}=\rho= 1-\alpha/n$. This implies T0) condition with parameters \e{x31}. Then suppose $B=B(x_0,R)$ and $x,x'\in B$, $y\in B^*$, are arbitrary points. One can check that
	\begin{equation}
		\left|\frac{1}{|x-y|^{n-\alpha}}-\frac{1}{|x'-y|^{n-\alpha}}\right|\lesssim \frac{R}{|x-y|^{n+1-\alpha}}.
	\end{equation}
Thus for $f\in L^1(\ZR^d)$ we obtain
\begin{equation}
	|I_\alpha f(x)-I_\alpha f(x')|\lesssim \sum_{k=0}^\infty \frac{1}{2^k}\cdot \frac{1}{|B(x_0,2^kR)|^\rho}\int_{B(x_0,2^kR)}|f(y)|dy\lesssim \langle f\rangle_B^*,
\end{equation}
which implies T1) condition.
\end{proof}
Weighted inequalities for operators \e{x30} and \e{x32} and more general potential operators have been studied in depth. See e.g. the works of Muckenhoupt and Wheeden \cite{Muck1}, Sawyer \cite{Saw1, Saw2}, Gabidzashvili and Kokilashvili \cite{Kok}, Sawyer and Wheeden \cite{Saw3}, and P\'{e}rez \cite{Per2,Per3}. Classical results of \cite{Muck1} assert that operators \e{x30}  and \e {x32} boundedly map $L^q(w^q)$ into $L^p(w^p)$ if and only if a weight $w$ satisfies 
\begin{equation}\label{x18}
	[w]_{A_{p,q}}=\sup_{B\in \ZB}\left(\frac{1}{\mu(B)}\int_Bw^q\right)\left(\frac{1}{\mu(B)}\int_Bw^{-p'}\right)^{q/p'}<\infty.
\end{equation}
The following sharp weighted estimates were proved by Lacey, Moem, P\'{e}rez and Torres in \cite{LaPe}:
\begin{align}
	&	\|I_\alpha(f)\|_{L^q(w^q)\to L^p(w^p)}\lesssim [w]_{A_{p,q}}^{(1-\alpha/n)\max\{1,p'/q\}},\label{x33}\\
	&	\|\MM_\alpha(f)\|_{L^q(w^q)\to L^p(w^p)}\lesssim [w]_{A_{p,q}}^{\frac{p'}{q}(1-\alpha/n)},
\end{align}
where parameters $q>p>1$ satisfy
\begin{equation}
	p^{-1}-q^{-1}=\frac{\alpha}{d}.
\end{equation}
Applying \trm{T7}, we obtain the following sparse domination inequality.
\begin{corollary}
	For any function $f\in L^1(\ZR^d)$ with $\supp f\subset B\in\ZB$ there is a sparse family of balls $\ZS$ such that
	\begin{equation}\label{x35}
		|I_\alpha f(x)|\lesssim C(\alpha)\sum_{G\in \ZS}\left\langle f\right\rangle_{G}\ZI_{G}(x),\text { a.e. } x\in B,
	\end{equation}
where 
\begin{equation}
	\langle f\rangle_G=\frac{1}{|G|^\rho}\int_G|f|,\quad \rho=1-\alpha/d.
\end{equation}
\end{corollary}
Consider an abstract measure space $(X,\mu)$ with a ball-basis $\ZB$ satisfying the Besicovitch condition.
\begin{definition}
	Let $\ZB$ be a family of sets of an arbitrary set $X$. We say a family of balls $\ZB$ satisfies the Besicovitch $D$-condition with a constant $D\in \ZN$, if for any collection $\ZG\subset \ZB$ one can find a subscollection $\ZG'\subset \ZG$ such that
	\begin{align*}
		&\bigcup_{G\in \ZG}G= \bigcup_{G\in \ZG'}G,\\
		&\sum_{ G\in \ZG'}\ZI_G(x)\le D.
	\end{align*}
	We say $\ZB$ is martingale system if $D=1$. 
\end{definition}
Consider a sparse operator
\begin{align}
	&\ZA_\ZS f(x)=\sum_{B\in \ZS}\langle f\rangle_B\cdot \ZI_B(x),\quad f\in L^1_{loc}(X)=L^1_{loc}(X,\ZR)\label{x16}\\
	&\text{ where }\langle f\rangle_B=\frac{1}{\mu(B)^\rho}\int_B|f|,\quad 0<\rho\le 1,
\end{align}
corresponding to a sparse collection of balls $\ZS\subset \ZB$. The method used in \cite{LaPe} can be applied to prove weighted estimate \e{x33} for sparse operators. Namely, we can state the following.
\begin{theorem}
	If a ball-basis $\ZB$ satisfies Besicovitch's condition and $w$ is an $A_{p,q}$ weight with parameters parameters $q>p>1$ satisfying 
	\begin{equation}
		p^{-1}-q^{-1}=1-\rho,
	\end{equation} then the sparse operator \e{x16} satisfies the sharp bound
	\begin{equation}\label{x34}
		\|\ZA\|_{L^q(w^q)\to L^p(w^p)}\lesssim [w]_{A_{p,q}}^{\rho \max\{1,p'/q\}}.
	\end{equation}
\end{theorem}
Having \e{x35} and \e{x34}, clearly we can obtain the weighted estimate \e{x33} of \cite{LaPe}. Such a conclusion could be interesting if we had a straightforward simple proof of \e{x34}.  In fact, the suggested proof of \e{x34} uses the same deep method of paper \cite{LaPe}. So it is not consistent to state the proof of \e{x34} in the present paper. Anyway the inequality \e{x34} is itself interesting, probably we will do it later in another note. Note that a straightforward proof of \e{x34} in the case of $\rho=1$ is known from \cite{Kar3}. To the best of our knowledge the method of \cite{Kar3} is not applicable to the general case. 
\section{Proofs of \trm{T2} and corollaries}
\begin{lemma}\label{LK1}
	If $A\subset B$ are balls and $x\notin B^*$, then $d(x,A)\ge\mu(B)$.
\end{lemma}
\begin{proof}
	Suppose to the contrary that $d(x,A)< \mu(B)$, which means there is a ball $A'\supset A$ such that $x\in A'$ and $\mu(A')<\mu(B)$. Then we get $A'\subset B^*$, which is in contradiction with $x\notin B^*$.
\end{proof}
\begin{lemma}\label{LK3}
	For any ball $B\in \ZB$ and a function $f\in L^r(X,\ZU)$ it holds the inequality
	\begin{equation}\label{a12}
		\langle f\rangle^*_{\#,B}\le\INF_{B}(\MM_{\#}(f))\lesssim \langle f\rangle^*_{\#,B}.
	\end{equation}
\end{lemma}
\begin{proof}
	The proof of the left hand side of the inequality is straightforward. To prove the right hand side inequality, we will need the following standard inequalities
	\begin{align}
		&\langle f\rangle_{\#,B}\le \langle f-c\rangle_B+\|f_B-c\|_\ZU\le 2\langle f-c\rangle_B,\quad c\in \ZU,\label{a33}\\
		&\langle f\rangle_{\#,B}\le 2\langle f-f_{B^*}\rangle_{B}\le 2\left(\frac{1}{\mu(B)}\int_{B^*}\|f-f_{B^*}\|_\ZU^r\right)^{1/r}\lesssim \langle f\rangle_{\#,B^*},\label{a35}
	\end{align}
	where $f\in L^r(X,\ZU)$ and $B$ is an arbitrary ball. For any $x\in B$ there exists a ball $B(x)\ni x$ such that
	\begin{equation}\label{a36}
		\langle f\rangle_{\#,B(x)}>\MM_{\#}f(x)/2\ge \INF_{B}(\MM_{\#}(f))/2=\lambda.
	\end{equation}
	Applying \lem{L1-1}, we find a sequence of pairwise disjoint balls $\{B_k\}\subset \{B(x):\, x\in B\}$ such that $\cup_kB_k^*\supset B$. If some $B_k$ satisfies $\mu(B_k)>\mu(B)$, then we have $B\subset B_k^*$ and, using \e{a35}, we get
	\begin{align*}
		\langle f\rangle^*_{\#,B}\ge \langle f\rangle_{\#,B_k^*}&\gtrsim\langle f\rangle_{\#,B_k}>\lambda.
	\end{align*}
	If $\mu(B_k)\le \mu(B)$ for every $k$, then  $\cup_kB_k\subset B^*$. Therefore by \e{a33}, \e{a36} and the pairwise disjointness of $B_k$ we obtain
	\begin{align*}
		\langle f\rangle^*_{\#,B}&\ge \langle f\rangle_{\#,B^*}\ge \left(\frac{1}{\mu(B^*)}\sum_k\int_{B_k}\|f-f_{B^*}\|_\ZU^r\right)^{1/r}\\
		&\gtrsim \left(\frac{1}{\mu(B^*)}\sum_k\int_{B_k}\|f-f_{B_k}\|_\ZU^r\right)^{1/r}\\
		&=\left(\frac{1}{\mu(B^*)}\sum_k\mu(B_k)(\langle f\rangle_{\#,B_k})^r\right)^{1/r}\\
		&\ge \lambda\left(\frac{1}{\mu(B^*)}\sum_k\mu(B_k)\right)^{1/r}\\
		&\gtrsim\lambda\left(\frac{1}{\mu(B)}\sum_k\mu(B_k^*)\right)^{1/r}\ge \lambda,
	\end{align*}
	which gives the right inequality in \e{a12}.
\end{proof}
\begin{lemma}\label{LK2}
	Let $\omega$ satisfy \e{y7} and a family $\phi=\{\phi_B:\, B\in \ZB\}$ be $\omega$-regular. Then for any function $f\in L^r(X,\ZU)$ and balls $A\subset B$ we have
	\begin{equation}\label{y2}
		\int_{X}\|f-f_{\phi_A}\|_\ZU\phi_B\lesssim (\|\omega\|+\log(\mu(B)/\mu(A)) )\cdot \langle f\rangle_{\#,A}^*.
	\end{equation}
\end{lemma}
\begin{proof}
	Applying the doubling condition we find a sequence of balls
	\begin{align}
		&B=B_0\subset B_1\subset \ldots\subset B_n\subset \ldots, \\ 
		&2\mu(B_{n-1})\le \mu(B_{n})\le \eta \mu(B_{n-1}).
	\end{align}
	For $x\in B_{n}^*\setminus B_{n-1}^*$ by \lem{LK1} we have 
	\begin{equation}
		d(x,B)\ge  \mu(B_{n-1}^*)\gtrsim \mu(B_n^*)\gtrsim 2^n\mu(B).
	\end{equation} 
	Thus, using \e{y4}, we obtain
	\begin{equation}
		\phi_B(x)\lesssim \frac{\ZI_{B^*}(x)}{\mu(B^*)}+\sum_{n=1}^\infty\frac{\omega(2^{-n})}{\mu(B_n^*)}\ZI_{B_{n}^*\setminus B_{n-1}^*}(x).
	\end{equation}
	Then by \lem{L10} we have 
	\begin{align}
		\|f_{B_{n}^*}-f_A\|_\ZU&\le \|f_{B_{n}^*}-f_B\|_\ZU+\|f_{B}-f_A\|_\ZU\\
		&\lesssim (n+\log(\mu(B)/\mu(A)) )\langle f\rangle_{\#,A}^*. 
	\end{align}
	Thus we obtain
	\begin{align}
		\int_{X}\|f-f_A\|_\ZU\phi_B&\lesssim \frac{1}{\mu(B^*)}\int_{B^*}\|f-f_A\|_\ZU+\sum_{n=1}^\infty \frac{\omega(2^{-n})}{\mu(B_n^*)}\int_{B_{n}^*\setminus B_{n-1}^*}\|f-f_A\|_\ZU\\
		&\le\frac{1}{\mu(B^*)}\int_{B^*}\|f-f_{B^*}\|_\ZU+\sum_{n=1}^\infty \frac{\omega(2^{-n})}{\mu(B_{n}^*)}\int_{B_{n}^*}\|f-f_{B_{n}^*}\|_\ZU\\
		&\qquad+\|f_{B^*}-f_A\|_\ZU+\sum_{n=1}^\infty \omega(2^{-n})\|f_{B_{n}^*}-f_A\|_\ZU\\
		&\lesssim (\|\omega\|+\log(\mu(B)/\mu(A)))\cdot  \langle f\rangle_{\#,A}^*.\label{k5}
	\end{align}
	Applying this inequality with $B=A$ we get
	\begin{equation}\label{k6}
		\|f_{\phi_A}-f_A\|_\ZU=\left\|\int_{X}(f-f_A)\phi_A\right\|_\ZU\lesssim  \|\omega\|\cdot \langle f\rangle_{\#,A}^*
	\end{equation}
	then, combining \e{k5} and \e{k6}, we obtain
	\begin{align*}
		\int_{X}\|f-f_{\phi_A}\|_\ZU\phi_B&\le \int_{X}\|f-f_A\|_\ZU\phi_B+\|f_{\phi_A}-f_A\|_\ZU\\
		&\lesssim (\|\omega\|+\log(\mu(B)/\mu(A)))\cdot \langle f\rangle_{\#,A}^*.
	\end{align*}
\end{proof}
\begin{remark}\label{RK1}
	We will also use the inequality 
	\begin{equation}\label{k7}
		\|f_{\phi_A}-f_{\phi_B}\|_\ZU\lesssim(\|\omega\|+\log(\mu(B)/\mu(A)) )\cdot \langle f\rangle_{\#,A}^*,
	\end{equation}
	which immediately follows from \e{y2}.
\end{remark}

\begin{proof}[Proof of \trm{T2}]
	Let $f\in L^r(X,\ZU)$ be a nontrivial function and $B$ be an arbitrary ball. Set $g=(f-f_{\phi_B})\cdot \ZI_{[B]}$ and  $E=E_{B,\lambda}=\{y\in B:\, \MM g(y)\le \lambda\}$, where $\MM$ is the maximal function in \e{1-1} with parameters $\varrho=\rho=1/r$, $r\ge 1$. According to the weak-$L^r$ bound of $\MM$ we have
	\begin{equation}
		\mu(B\setminus E)=\mu\{y\in B:\, \MM g(y)>\lambda\}\lesssim  \frac{1}{\lambda^r}\cdot \int_{[B]}|g|^r.
	\end{equation}
	So for an appropriate number $\lambda\sim (1-\alpha)^{-1/r}\langle g\rangle_{[B]}$	we have $\mu(B\setminus E)< (1-\alpha)\mu(B)$ and therefore,
	\begin{equation}\label{b2}
		\mu(E)>\alpha \mu(B).
	\end{equation} 
	Applying left hand side of \e{y4}, as well as \e{y2} and \e{k7}, for any $y\in E$ we can write
	\begin{align}
		&\MM g(y)\lesssim (1-\alpha)^{-1/r}\langle g\rangle_{[B]}=(1-\alpha)^{-1/r}\langle f-f_{\phi_B}\rangle_{[B]}\label{b1}\\
		&\qquad\quad\lesssim (1-\alpha)^{-1/r}\left(\langle f-f_{\phi_{[B]}}\rangle_{\phi_{[B]}}+\|f_{\phi_B}-f_{\phi_{[B]}}\|_\ZU\right)\\
		&\qquad\quad\, \lesssim (1-\alpha)^{-1/r}\|\omega\|\langle f\rangle^*_{\#,B},\quad y\in E.
	\end{align}
	Take arbitrary points $x,x'\in E$. Without loss of generality we can suppose that $\MM^{\phi,\ZG} f(x)\ge \MM^{\phi,\ZG} f(x')$. For any $\delta >0$ there is a ball $A\in \ZG(x)$ (so $A\ni x$)  such that
	\begin{equation}\label{k1}
		\MM^{\phi,\ZG} f(x)\le  \langle f\rangle_{A}+\delta.
	\end{equation}
	If $\mu(A)>\mu(B)$, then we have $x'\in B\subset [A]$ and so there is a ball $\bar A\in \ZB(x')$ such that 
	\begin{equation}\label{k2}
		\bar A\supset [A],\quad \mu(\bar A)\le \eta \mu([A])\lesssim \mu(A).
	\end{equation}
	Then by definition of $\MM^{\phi,\ZG}$ we can write
	\begin{equation}\label{k3}
		\MM^{\phi,\ZG} f(x')\ge \langle f\rangle_{\phi_{\bar A}}.
	\end{equation}
	Using left hand side of \e{y4} together with relations \e{y2}, \e{k1}, \e{k2} and \e{k3}, we obtain
	\begin{align}
		\MM^{\phi,\ZG} f(x)-\MM^{\phi,\ZG} f(x')&\le \langle f\rangle_{A}-\langle f\rangle_{\phi_{\bar A}}+\delta\\
		&\le \langle f-f_{\phi_{\bar A}}\rangle_{A}+|f_{\phi_{\bar A}}|+\langle f-f_{\phi_{\bar A}}\rangle_{\phi_{\bar A}}-|f_{\phi_{\bar A}}|+\delta\\
		&\lesssim  \langle f-f_{\phi_{\bar A}}\rangle_{\phi_{\bar A}}+\langle f-f_{\phi_{\bar A}}\rangle_{\phi_{\bar A}}+\delta\\
		&\lesssim \|\omega\|\langle f\rangle^*_{\#,B}+\delta.\label{a10}
	\end{align}
	If $\mu(A)\le \mu(B)$, then $x'\in A\subset  [B]$ and 
	\begin{equation}\label{k4}
		\MM^{\phi,\ZG} f(x')\ge \langle f\rangle_{\phi_{\bar B}}
	\end{equation}
	for a ball $\bar B\in \ZB(x')$ with $\mu(\bar B)\le \eta \mu([B])\lesssim\mu(B)$. Thus, using also \e{b1} and \e{k1}, we obtain
	\begin{align}\label{a11}
		\MM^{\phi,\ZG} f(x)-\MM^{\phi,\ZG} f(x')&\le \langle f\rangle_{A}-\langle f\rangle_{\phi_{\bar B}}+\delta\\
		&\le \langle f-f_{\phi_{\bar B}}\rangle_{A}+\|f_{\phi_{\bar B}}\|_\ZU+\langle f-f_{\phi_{\bar B}}\rangle_{{\phi_{\bar B}}}-\|f_{\phi_{\bar B}}\|_\ZU+\delta\\
		&\le \langle f-f_{\phi_{B}}\rangle_{A}+\|f_{\phi_B}-f_{\phi_{\bar B}}\|_\ZU+\langle f-f_{\phi_{\bar B}}\rangle_{{\phi_{\bar B}}}+\delta\\
		&\lesssim \langle g\rangle_{A}+ \|\omega\|\langle f\rangle^*_{\#,B}+\delta\\
		&\le \MM g(x)+ \|\omega\|\langle f\rangle^*_{\#,B}+\delta\\
		&\lesssim  (1-\alpha)^{-1/r} \|\omega\|\langle f\rangle^*_{\#,B}+\delta.
	\end{align}
	Since $\delta$ can be arbitrary small, from \e{a10} and \e{a11} we conclude
	\begin{equation*}
		|\MM^{\phi,\ZG} f(x)-\MM^{\phi,\ZG} f(x')|\lesssim (1-\alpha)^{-1/r} \|\omega\|\langle f\rangle^*_{\#,B}, \quad x,x'\in E.
	\end{equation*}
	This implies
	\begin{equation}\label{b3}
		\OSC_E(\MM^{\phi,\ZG} f)\lesssim (1-\alpha)^{-1/r} \|\omega\|\langle f\rangle^*_{\#,B}. 
	\end{equation}
	Combining \e{b2} and  \e{b3}, we deduce \e{2}, completing the proof of theorem.
\end{proof}
The proofs of corollaries \ref{KC1} and \ref{KC2} are similar to the analogous corollaries \e{T8}. One just need to combine the result of \trm{T2} with \pro {T3}.
\bigskip

\noindent
\textbf{Declaration of competing interest}

\bigskip
The Authors have no conflicts of interest to declare that are relevant to the content
of this article.

\bigskip
\noindent
\textbf{Data availability}

\bigskip
Data sharing is not applicable to this article as no datasets were generated or analysed
during the current study.

\end{document}